\documentclass{amsart}

\usepackage{thesis} 

\begin{document}

\title{\MakeUppercase{\hammockstitle}}
\date{\today}
\author{\textsc{Aaron Mazel-Gee}}

\begin{abstract}
We study the \textit{homotopy theory} of $\infty$-categories enriched in the $\infty$-category $s\S$ of simplicial spaces.  That is, we consider $s\S$-enriched $\infty$-categories as presentations of ordinary $\infty$-categories by means of a ``local'' geometric realization functor $\CatsS \ra \Cati$, and we prove that their homotopy theory presents the $\infty$-category of $\infty$-categories, i.e.\! that this functor induces an equivalence $\loc{\CatsS}{\bW_\DK} \xra{\sim} \Cati$ from a localization of the $\infty$-category of $s\S$-enriched $\infty$-categories.

Following Dwyer--Kan, we define a \textit{hammock localization} functor from relative $\infty$-categories to $s\S$-enriched $\infty$-categories, thus providing a rich source of examples of $s\S$-enriched $\infty$-categories.  Simultaneously unpacking and generalizing one of their key results, we prove that given a relative $\infty$-category admitting a \textit{homotopical three-arrow calculus}, one can explicitly describe the hom-spaces in the $\infty$-category presented by its hammock localization in a much more explicit and accessible way.

As an application of this framework, we give sufficient conditions for the Rezk nerve of a relative $\infty$-category to be a (complete) Segal space, generalizing joint work with Low.
\end{abstract}

\maketitle

\papernum{4}

\setcounter{tocdepth}{1}
\tableofcontents

\setcounter{section}{-1}

\section{Introduction}

\subsection{Introducing (even more) homotopy theory}\label{subsection intro to hammocks}

In their groundbreaking papers \cite{DKSimpLoc} and \cite{DKCalc}, Dwyer--Kan gave the first presentation of \textit{the $\infty$-category of $\infty$-categories}, namely the category
$ \strcat_{s\Set} $
of \bit{categories enriched in simplicial sets}: in modern language, every $s\Set$-enriched category has an \textit{underlying $\infty$-category}, and this association induces an equivalence
\[ \loc{\strcat_{s\Set}}{\bW_\DK} \xra{\sim} \Cati \]
from the ($\infty$-categorical) localization of the category $\strcat_{s\Set}$ at the subcategory $\bW_\DK \subset \strcat_{s\Set}$ of \textit{Dwyer--Kan weak equivalences} to the $\infty$-category $\Cati$ of $\infty$-categories.  Moreover, Dwyer--Kan provided a method of \bit{``introducing homotopy theory''} into a category $\R$ equipped with a subcategory $\bW \subset \R$ of weak equivalences, namely their \textit{hammock localization} functor $\hamd : \strrelcat \ra \strcat_{s\Set}$ of \cite{DKCalc}.

In this paper, we set up an analogous framework in the setting of \textit{$\infty$-categories}: we prove that the $\infty$-category $\CatsS$ of \bit{$\infty$-categories enriched in simplicial spaces} likewise models the $\infty$-category of $\infty$-categories via an equivalence
\[ \loc{\CatsS}{\bW_\DK} \xra{\sim} \Cati , \]
and we define a \textit{hammock localization} functor
$ \ham : \RelCati \ra \CatsS $
which likewise provides a method of \bit{``introducing (even more) homotopy theory''} into relative $\infty$-categories.
We moreover prove the following two results -- the first generalizing a theorem of Dwyer--Kan, the second generalizing joint work with Low (see \cite{LowMG}).

\begin{thm*}[\ref{calculus gives reduction}]
Given a relative $\infty$-category $(\R,\bW)$ admitting a \bit{homotopical three-arrow calculus}, the hom-spaces in the underlying $\infty$-category of its hammock localization admit a canonical equivalence
\[ \word{3}(x,y)^\gpd \xra{\sim} \left| \homhamR(x,y) \right| \]
from the groupoid completion of the $\infty$-category of three-arrow zigzags $x \lwe \bullet \ra \bullet \lwe y$ in $(\R,\bW)$.
\end{thm*}

\begin{thm*}[\ref{calculus result}]
Given a relative $\infty$-category $(\R,\bW)$, its \bit{Rezk nerve} \[ \NerveRezki(\R,\bW) \in s\S \]
\begin{itemizesmall}
\item is a \bit{Segal space} if $(\R,\bW)$ admits a homotopical three-arrow calculus, and
\item is moreover a \bit{complete Segal space} if moreover $(\R,\bW)$ is saturated and satisfies the two-out-of-three property.
\end{itemizesmall}
\end{thm*}

\noindent (The notion of a \textit{homotopical three-arrow calculus} is a minor variant on Dwyer--Kan's ``homotopy calculus of fractions'' (see \cref{define calculus}).  Meanwhile, the \textit{Rezk nerve} is a straightforward generalization of Rezk's ``classification diagram'' construction, which we introduced in \cite{MIC-rnerves} and proved computes the $\infty$-categorical localization (see \cite[Theorem 3.8 and Corollary 3.12]{MIC-rnerves}).)

\begin{rem}
In \cref{compare with barwick's theory of enriched infty-cats}, we show how our notion of ``$s\S$-enriched $\infty$-category'' fits with the corresponding notion coming from Barwick's theory of distributors.
\end{rem}

\begin{rem}\label{intro rem more details than DK}
Many of the original Dwyer--Kan definitions and proofs are quite point-set in nature.  However, when working $\infty$-categorically, it is essentially impossible to make such ad hoc constructions.  Thus, we have no choice but to be both much more careful and much more precise in our generalization of their work.\footnote{For example, our proof of \cref{calculus gives reduction} spans nearly four pages whereas the proof of \cite[Proposition 6.2(i)]{DKCalc} (which it generalizes) is just half a page long, and our proof of \cref{hammocks are invt under w.e.} is nearly three pages whereas the proof of \cite[Proposition 3.3]{DKCalc} (which it generalizes) is not even provided.}  We find Dwyer--Kan's facility with universal constructions (displayed in that proof and elsewhere) to be really quite impressive, and we hope that our elaboration on their techniques will be pedagogically useful.  Broadly speaking, our main technique is to \textit{corepresent} higher coherence data.
\end{rem}

\subsection{Conventions}

\partofMIC

\tableofcodenames

\examplecodename

\citelurie \ \luriecodenames

\butinvt \ \seeappendix

\subsection{Outline}

We now provide a more detailed outline of the contents of this paper.

\begin{itemize}
\item In \cref{section SsS and sS-enr infty-cats}, we introduce the $\infty$-category $\CatsS$ of $\infty$-categories enriched in simplicial spaces, as well as an auxiliary $\infty$-category $\CSS$ of \textit{Segal simplicial spaces}.  We endow both of these with subcategories of \textit{Dwyer--Kan weak equivalences}, and prove that the resulting relative $\infty$-categories both model the $\infty$-category $\Cati$ of $\infty$-categories.
\item In \cref{section zigzags and hammocks}, we define the $\infty$-categories of zigzags in a relative $\infty$-category $(\R,\bW)$ between two objects $x,y \in \R$, and use these to define the \textit{hammock simplicial spaces} $\homhamR(x,y)$, which will be the hom-simplicial spaces in the hammock localization $\ham(\R,\bW)$.
\item In \cref{section three-arrow calculi}, we define what it means for a relative $\infty$-category to admit a homotopical three arrow calculus, and we prove the first of the two results stated above.
\item In \cref{section hammocks}, we finally construct the hammock localization functor on relative $\infty$-categories, and we explore some of its basic features.
\item In \cref{section fractions redux}, we prove the second of the two results stated above.
\end{itemize}

\subsection{Acknowledgments}

We would like to thank David Ayala, Marc Hoyois, Tyler Lawson, and Zhen Lin Low for their helpful input.  We also gratefully acknowledge the financial support from the NSF graduate research fellowship program (grant DGE-1106400) provided during the time that this paper was written.

\section{Segal spaces, Segal simplicial spaces, and $s\S$-enriched $\infty$-categories}\label{section SsS and sS-enr infty-cats}

In this section, we develop the theory -- and the \textit{homotopy} theory -- of two closely related flavors of higher categories whose hom-objects lie in the symmetric monoidal $\infty$-category $(s\S,\times)$ of simplicial spaces equipped with the cartesian symmetric monoidal structure.  By ``homotopy theory'', we mean that we will endow the $\infty$-categories of these objects with \textit{relative $\infty$-category} structures, whose weak equivalences are created by ``local'' (i.e.\! hom-object-wise) geometric realization.  These therefore constitute ``many-object'' elaborations on the Kan--Quillen relative $\infty$-category $(s\S, \bW_\KQ)$, whose weak equivalences are created by geometric realization (see \cref{sspaces:kan--quillen model structure on sspaces}).  A key source of such objects will be the \textit{hammock localization} functor, which we will introduce in \cref{section hammocks}.


This section is organized as follows.
\begin{itemize}
\item In \cref{subsection segal spaces}, we recall some basic facts regarding \textit{Segal spaces}.
\item in \cref{subsection segal sspaces}, we introduce \textit{Segal simplicial spaces} and define the essential notions for ``doing (higher) category theory'' with them.
\item In \cref{subsection sS-enriched infty-cats}, we introduce their full (in fact, coreflective) subcategory of \textit{simplicio-spatially-enriched} (or simply \textit{$s\S$-enriched}) \textit{$\infty$-categories}.  These are useful since they can more directly be considered as ``presentations of $\infty$-categories''.
\item In \cref{subsection SsS and Cat_sS present infty-cats}, we prove that freely inverting the \textit{Dwyer--Kan weak equivalences} among either the Segal simplicial spaces or the $s\S$-enriched $\infty$-categories yields an $\infty$-category which is canonically equivalent to $\Cati$ itself.  We also contextualize both of these sorts of objects with respect to Barwick's theory of enriched $\infty$-categories, and provide some justification for our interest in them.
\end{itemize}

\subsection{Segal spaces}\label{subsection segal spaces}

We begin this section with the following recollections.  This subsection exists mainly in order to set the stage for the remainder of the section; we refer the reader seeking a more thorough discussion either to the original paper \cite{RezkCSS} (which uses model categories) or to \cite[\sec 1]{LurieGoo} (which uses $\infty$-categories).

\begin{defn}\label{define segal spaces}
The $\infty$-category of \bit{Segal spaces} is the full subcategory $\SS \subset s\S$ of those simplicial spaces satisfying the \textit{Segal condition}.  These sit in a left localization adjunction
\[ \begin{tikzcd}[column sep=2cm]
s\S \arrow[transform canvas={yshift=0.7ex}]{r}{\leftloc_\SS}[swap, transform canvas={yshift=0.2ex}]{\scriptstyle \bot} \arrow[transform canvas={yshift=-0.7ex}, hookleftarrow]{r}[swap]{\forget_\SS}
& \SS ,
\end{tikzcd} \]
which factors the left localization adjunction $\leftloc_\CSS \adj \forget_\CSS$ of \cref{rnerves:define CSSs} in the sense that we obtain a pair of composable left localization adjunctions
\[ \begin{tikzcd}[column sep=2cm]
s\S
\arrow[transform canvas={yshift=0.7ex}]{r}{\leftloc_\SS}[swap, transform canvas={yshift=0.2ex}]{\scriptstyle \bot} \arrow[transform canvas={yshift=-0.7ex}, hookleftarrow]{r}[swap]{\forget_\SS}
& \SS
\arrow[transform canvas={yshift=0.7ex}]{r}{\leftloc_\CSS}[swap, transform canvas={yshift=0.2ex}]{\scriptstyle \bot} \arrow[transform canvas={yshift=-0.7ex}, hookleftarrow]{r}[swap]{\forget_\CSS}
& \CSS .
\end{tikzcd} \]
(This follows easily from \cite[Theorems 7.1 and 7.2]{RezkCSS}, or alternatively more-or-less follows from \cite[Remark 1.2.11]{LurieGoo}.)
\end{defn}

In order to make a few basic observations, it will be convenient to first introduce the following.

\begin{defn}\label{define 0th coskel and pb of simp obj}
Suppose that $\C \in \Cati$ admits finite products.  Then, we define the \bit{$0\th$ coskeleton} of an object $c \in \C$ (or perhaps more standardly, of the corresponding constant simplicial object $\const(c) \in s\C$) to be the simplicial object selected by the composite
\[ \bD^{op} \hookra (s\Set)^{op} \xra{((-)_0)^{op}} \Set^{op} \hookra \S^{op} \xra{- \cotensoring c} \C . \]
This assembles to a functor
\[ \C \xra{(-)^{\times(\bullet+1)}} s\C \]
which, as the notation suggests, is given in degree $n$ by $c \mapsto c^{\times (n+1)}$.  This sits in an adjunction
\[ (-)_0 : s\C \adjarr \C : (-)^{\times(\bullet+1)} , \]
which we refer to as the \bit{$0\th$ coskeleton adjunction} for $\C$.  Using this, given a simplicial object $Z \in s\C$ and a map $Y \xra{\varphi} Z_0$ in $\C$, we define the \bit{pullback} of $Z$ along $\varphi$ to be the fiber product
\[ \varphi^*(Z) = \lim \left( \begin{tikzcd}[column sep=1.5cm]
& Z_\bullet \arrow{d} \\
Y^{\times(\bullet+1)} \arrow{r}[swap]{\varphi^{\times(\bullet + 1)}} & (Z_0)^{\times(\bullet+1)}
\end{tikzcd} \right) \]
in $s\C$, where the vertical map is the component at the object $Z \in s\C$ of the unit of the $0\th$ coskeleton adjunction.  In particular, note that we have a canonical equivalence $(\varphi^*(Z))_0 \simeq Y$ in $\C$.
\end{defn}

\begin{rem}\label{any SS pulled back from its CSS-localizn}
Suppose that $Y \in \SS$, and let us write $Y \xra{\lambda} \leftloc_\CSS(Y)$ for its localization map.  Then, the map $Y_0 \xra{\lambda_0} \leftloc_\CSS(Y)_0$ is a surjection, and moreover we have a canonical equivalence
\[ Y \simeq (\lambda_0)^*(\leftloc_\CSS(Y)) \]
in $\SS \subset s\S$.  (The first claim follows from \cite[Theorem 7.7 and Corollary 6.5]{RezkCSS}, while the second claim follows from combining \cite[Definition 1.2.12(b) and Theorem 1.2.13(2)]{LurieGoo} with the Segal condition for $Y \in s\S$.)  From here, it follows easily that we have an equivalence
\[ \SS \simeq \lim \left( \begin{tikzcd}
& \Fun^{\textup{surj}}([1],\Cati) \arrow{d}{s} \\
\S \arrow[hookrightarrow]{r} & \Cati
\end{tikzcd} \right) , \]
where $\Fun^{\textup{surj}}([1],\Cati) \subset \Fun([1],\Cati)$ denotes the full subcategory on those functors $[1] \ra \Cati$ that select surjective maps $\C \ra \D$.  From this viewpoint, the left localization $\leftloc_\CSS : \SS \ra \CSS$ is then just the composite functor
\[ \SS \hookra \Fun^{\textup{surj}}([1],\Cati) \xra{t} \Cati \xra[\sim]{\Nervei} \CSS . \]
Thus, one might think of $\SS$ as ``the $\infty$-category of surjectively marked $\infty$-categories'' (where by ``surjectively marked'' we of course mean ``equipped with a surjective map from an $\infty$-groupoid'').
\end{rem}

\begin{rem}\label{rem segal spaces vs strict cats}
Continuing with the observations of \cref{any SS pulled back from its CSS-localizn}, note that the category $\strcat$ of \textit{strict} 1-categories can be recovered as a limit
\[ \begin{tikzcd}[column sep=1.5cm, row sep=1.5cm]
\strcat \arrow{dd} \arrow{rrr} & & & \Cat \arrow[hook]{d} \\
& \SS \arrow[hook]{r} \arrow{d} & \Fun^{\textup{surj}}([1],\Cati) \arrow{r}[swap]{t} \arrow{d}{s} & \Cati \\
\Set \arrow[hook]{r} & \S \arrow[hook]{r} & \Cati
\end{tikzcd} \]
in $\Cati$ (in which the square is already a pullback).  (In fact, the induced map $\strcat \ra \SS$ itself fits into the defining pullback square
\[ \begin{tikzcd}[row sep=1.25cm, column sep=1.25cm]
\strcat \arrow{r} \arrow[hook]{d}[swap]{\Nerve} & \SS \arrow[hook]{d}{\forget_\SS} \\
s\Set \arrow[hook]{r} & s\S
\end{tikzcd} \]
in $\Cati$.)  We can therefore consider the $\infty$-category $\SS$ of Segal spaces as a close cousin of the 1-category $\strcat$ of strict categories, with the caveat that objects of $\strcat$ must be surjectively marked by a \textit{discrete} space.
\end{rem}

\begin{rem}\label{extract hom-spaces from a SS}
Suppose that $Y \in \SS$.  Then, we can compute hom-spaces in the $\infty$-category
\[ \C = \Nervei^{-1}(\leftloc_\CSS(Y)) \in \Cati \]
as follows.
Any pair of objects $x,y \in \C$ can be considered as defining a pair of points
\[ x,y \in \C^\simeq \simeq \Nervei(\C)_0 \simeq \leftloc_\CSS(Y)_0 . \]
Since the map $Y_0 \ra \leftloc_\CSS(Y)_0$ is a surjection, these admit lifts $\tilde{x},\tilde{y} \in Y_0$.  Then, we have a composite equivalence
\[ \hom_\C(x,y) \simeq \lim \left( \begin{tikzcd}
 & \Nervei(\C)_1 \arrow{d}{(s,t)} \\
 \pt_\S \arrow{r}[swap]{(x,y)} & \Nervei(\C)_0 \times \Nervei(\C)_0
\end{tikzcd} \right)
\simeq \lim \left( \begin{tikzcd}
 & Y_1 \arrow{d}{(s,t)} \\
 \pt_\S \arrow{r}[swap]{(\tilde{x},\tilde{y})} & Y_0 \times Y_0
\end{tikzcd} \right) \]
by Remarks \Cref{rnerves:extract hom-spaces from a CSS} \and \ref{any SS pulled back from its CSS-localizn}.  (In particular, we can compute the hom-space $\hom_\C(x,y)$ using \textit{any} choices of lifts $\tilde{x},\tilde{y} \in Y_0$.)
\end{rem}

\subsection{Segal simplicial spaces}\label{subsection segal sspaces}

We now turn from the $\S$-enriched context to the $s\S$-enriched context.

\begin{defn}\label{define Segal sspaces}
We define the $\infty$-category of \bit{Segal simplicial spaces} to be the full subcategory $\SsS \subset s(s\S)$ of those simplicial objects in $s\S$ which satisfy the Segal condition.  These sit in a left localization adjunction $s(s\S) \adjarr \SsS$ by the adjoint functor theorem (Corollary T.5.5.2.9).
\end{defn}

\begin{rem}\label{rem SsS as simplicial cats}
In light of \cref{rem segal spaces vs strict cats}, we can consider the $\infty$-category $\SsS$ of Segal simplicial spaces as being a homotopical analog of the 1-category $s\strcat = \Fun(\bD^{op},\strcat)$ of simplicial categories.  The subcategory $\strcat_{s\Set} \subset s\strcat$ of $s\Set$-enriched categories then corresponds to the full subcategory on those Segal simplicial spaces $\C_\bullet \in \SsS$ such that the ``levelwise $0\th$ space'' object $(\C_\bullet)_0 \in s\S$ is \textit{constant}.
\end{rem}

\begin{defn}\label{define space of objects and hom-sspaces}
For any $\C_\bullet \in \SsS$, we define the \bit{space of objects} of $\C_\bullet$ to be the space
\[ (\C_0)_0 \simeq \hom_{s\S}(\pt_{s\S},\C_0) \in \S , \]
and for any $x, y \in (\C_0)_0$, we define the \bit{hom-simplicial space} from $x$ to $y$ in $\C_\bullet$ to be the pullback
\[ \ul{\hom}_{\C_\bullet}(x,y) = \lim \left( \begin{tikzcd}
& \C_1 \arrow{d}{(s,t)} \\
\pt_{s\S} \arrow{r}[swap]{(x,y)} & \C_0 \times \C_0
\end{tikzcd} \right) \]
in $s\S$.  We refer to the points of the space
\[ \enrhom_{\C_\bullet}(x,y)_0 \simeq \hom_{s\S}(\pt_{s\S},\enrhom_{\C_\bullet}(x,y)) \]
simply as \bit{morphisms} from $x$ to $y$.  The various hom-simplicial spaces of $\C_\bullet$ admit associative composition maps
\[ \ul{\hom}_{\C_\bullet}(x_0,x_1) \times \cdots \times \ul{\hom}_{\C_\bullet}(x_{n-1},x_n) \xra{\chi^{\C_\bullet}_{x_0,\ldots,x_n}} \ul{\hom}_{\C_\bullet}(x_0,x_n) \]
in $s\S$, which are obtained as usual via the Segal conditions.  For any $x \in (\C_0)_0$ there is an evident \bit{identity morphism} from $x$ to itself, denoted $\ul{\id}_x \in \enrhom_{\C_\bullet}(x,x)_0$, which behaves as expected under these composition maps.
\end{defn}

\begin{defn}
Given any $\C_\bullet \in \SsS$ and any pair of objects $x,y \in (\C_0)_0$, we say that two morphisms
\[ \pt_{s\S} \rra \enrhom_{\C_\bullet}(x,y) \]
are \bit{simplicially homotopic} if the induced maps
\[ \pt_\S \rra | \enrhom_{\C_\bullet}(x,y)| \]
are equivalent (i.e.\! select points in the same path component of the target).  We then say that a morphism $f \in \enrhom_{\C_\bullet}(x,y)_0$ is a \bit{simplicial homotopy equivalence} if there exists a morphism $g \in \enrhom_{\C_\bullet}(y,x)_0$ such that the composite morphisms
\[ \chi^{\C_\bullet}_{x,y,x}(f,g) \in \enrhom_{\C_\bullet}(x,x) \]
and
\[ \chi^{\C_\bullet}_{y,x,y}(g,f) \in \enrhom_{\C_\bullet}(y,y) \]
are simplicially homotopic to the respective identity morphisms.
\end{defn}

Now, the objects of $\SsS$ will indeed be ``presentations of $\infty$-categories'', but maps between them which are not equivalences may nevertheless induce equivalences between the $\infty$-categories that they present.  We therefore introduce the following notion.

\begin{defn}\label{DK equivces in SsS}
A map $\C_\bullet \xra{\varphi_\bullet} \D_\bullet$ in $\SsS$ is called a \bit{Dwyer--Kan weak equivalence} if
\begin{itemize}
\item it is \bit{weakly fully faithful}, i.e.\! for all pairs of objects $x,y \in (\C_0)_0$ the induced map
\[ \left| \ul{\hom}_{\C_\bullet}(x,y) \right| \ra \left| \ul{\hom}_{\D_\bullet}(\varphi(x),\varphi(y)) \right| \]
is an equivalence in $\S$, and
\item it is \bit{weakly surjective}, i.e.\! the map
\[ \pi_0((\C_0)_0) \xra{\pi_0((\varphi_0)_0)} \pi_0((\D_0)_0) \]
is surjective up to the equivalence relation on $\pi_0((\D_0)_0)$ generated by simplicial homotopy equivalence.
\end{itemize}
Such morphisms define a subcategory $\bW_\DK \subset \SsS$ containing all the equivalences and satisfying the two-out-of-three property, and we denote the resulting relative $\infty$-category by $\SsS_\DK = (\SsS,\bW_\DK) \in \RelCati$.
\end{defn}

\begin{rem}
Via the evident functor $\strcat_{s\Set} \ra \SsS$ (recall \cref{rem SsS as simplicial cats}), the subcategory of Dwyer--Kan weak equivalences $\bW_\DK^{\strcatsup_{s\Set}} \subset \strcat_{s\Set}$ of \cref{subsection intro to hammocks} (i.e.\! the subcategory of weak equivalences for the Bergner model structure) is pulled back from the subcategory $\bW^\SsS_\DK \subset \SsS$.
\end{rem}

\subsection{$s\S$-enriched $\infty$-categories}\label{subsection sS-enriched infty-cats}

In light of the discussion of \cref{subsection segal sspaces}, the natural guess for the sense in which a Segal simplicial space should be considered as a ``presentation of an $\infty$-category'' is via the levelwise geometric realization functor
\[ s(s\S) \xra{s(|{-}|)} s\S . \]
However, this operation does not preserve Segal objects: taking fiber products of simplicial spaces does not generally commute with taking their geometric realizations.  On the other hand, these two operations \textit{do} commute when the common target of the cospan is constant.  Hence, it will be convenient to restrict our attention to the following special class of objects.

\begin{defn}\label{define sspatial infty-cats}
We define the $\infty$-category of \bit{simplicio-spatially-enriched $\infty$-categories}, or simply of \bit{$s\S$-enriched $\infty$-categories}, to be the full subcategory $\CatsS \subset \SsS$ on those objects $\C_\bullet \in \SsS \subset s(s\S)$ such that $\C_0 \in s\S$ is constant.  We write
\[ \CatsS \xhookra{\forget_\CatsS} \SsS \]
for the defining inclusion.  Restricting the subcategory $\bW^\SsS_\DK \subset \SsS$ of Dwyer--Kan weak equivalences along this inclusion, we obtain a relative $\infty$-category $(\CatsS)_\DK = ( \CatsS,\bW_\DK) \in \RelCati$ (which also has the two-out-of-three property).
\end{defn}

\begin{lem}\label{lw geom realizn of sS-enr infty-cats gives segal spaces}
There is a canonical factorization
\[ \begin{tikzcd}
\CatsS \arrow[hookrightarrow]{r}{\forget_\CatsS} \arrow[dashed]{rrrd} & \SsS \arrow[hookrightarrow]{r} & s(s\S) \arrow{r}{s(|{-}|)} & s\S \\
& & & \SS \arrow[hookrightarrow]{u}
\end{tikzcd} \]
of the restriction of the levelwise geometric realization functor
\[ s(s\S) \xra{s(|{-}|)} s\S \]
to the subcategory $\CatsS \subset s(s\S)$ of $s\S$-enriched $\infty$-categories.
\end{lem}

\begin{proof}
This follows from Lemma A.5.5.6.17 (applied to the $\infty$-topos $\S$) and the fact that coproducts commute with connected limits.
\end{proof}

\begin{defn}\label{def geom realizn of sS-cats}
We denote simply by
\[ \CatsS \xra{|{-}|} \SS \]
the factorization of \cref{lw geom realizn of sS-enr infty-cats gives segal spaces}, and refer to it as the \bit{geometric realization} functor on $s\S$-enriched $\infty$-categories.
\end{defn}

\begin{defn}\label{def constant sS-enr infty-cat}
The composite inclusion
\[ \Cati \xra[\sim]{\Nervei} \CSS \xhookra{\forget_\CSS} s(\S) \xhookra{s(\const)} s(s\S) \]
clearly factors through the subcategory $\CatsS \subset \SsS \subset s(s\S)$.
We simply write
\[ \Cati \xra{\const} \CatsS \]
for this factorization, and refer to it as the \bit{constant $s\S$-enriched $\infty$-category} functor.  Thus, for an $\infty$-category $\C \in \Cati$, the simplicial object
\[ \const(\C)_\bullet \in \CatsS \subset s(s\S) \]
is given in degree $n$ by
\[ \const(\Nervei(\C)_n) \in s\S , \]
the constant simplicial space on the object
\[ \Nervei(\C)_n = \hom_\Cati([n],\C) \in \S . \]
This functor clearly participates in a commutative diagram
\[ \begin{tikzcd}
\Cati \arrow{r}{\const} \arrow{rrd}[sloped, pos=0.45]{\sim}[sloped, swap, pos=0.62]{\Nervei} & \CatsS \arrow{r}{|{-}|} & \SS \\
& & \CSS \arrow[hook]{u}[swap]{\forget_\CSS}
\end{tikzcd} \]
in $\Cati$.
\end{defn}

\begin{rem}\label{any SsS is DK-equivt to a sS-cat}
Suppose we are given a Segal simplicial space $\C_\bullet \in \SsS$ and a map $Z \xra{\varphi} (\C_0)_0$ in $\S$ to its space of objects.  Then, the canonical map
\[ \varphi^*(\C_\bullet) \ra \C_\bullet \]
is \textit{fully faithful} (in the \textit{$s\S$-enriched} sense): for any objects $x,y \in Z \simeq (\varphi^*(\C_\bullet)_0)_0$, the induced map
\[ \enrhom_{\varphi^*(\C_\bullet)}(x,y) \ra \enrhom_{\C_\bullet}(\varphi(x),\varphi(y)) \]
is already an equivalence in $s\S$ (instead of just being an equivalence upon geometric realization).  Of course, the map $\varphi^*(\C_\bullet) \ra \C_\bullet$ is therefore in particular weakly fully faithful as well.  As we can always choose our original map $Z \xra{\varphi} (\C_0)_0$ so that the induced map $\varphi^*(\C_\bullet) \ra \C_\bullet$ is additionally weakly surjective (e.g.\! by taking $\varphi$ to be a surjection), it follows that any Segal simplicial space admits a Dwyer--Kan weak equivalence from a $s\S$-enriched category; indeed, we can even arrange to have $Z \in \Set \subset \S$.
\end{rem}

Improving on \cref{any SsS is DK-equivt to a sS-cat}, we now describe a \textit{universal} way of extracting a $s\S$-enriched $\infty$-category from a Segal simplicial space.

\begin{defn}\label{spatialization}
We define the \bit{spatialization} functor $\spat : \SsS \ra \CatsS$ as follows.\footnote{The word ``spatialization'' is meant to indicate that the $0\th$ object of its output will lie in the subcategory $\S \subset s\S$ of constant simplicial spaces.}  Any $\C_\bullet \in \SsS$ gives rise to a natural map
\[ \const((\C_0)_0) \xra{\varepsilon} \C_0 \]
in $s\S$, the component at $C_0 \in s\S$ of the counit of the right localization adjunction $\const : \S \adjarr s\S : \lim$.  The spatialization of $\C_\bullet$ is then the pullback
\[ \spat(\C_\bullet) = \varepsilon^*(\C_\bullet) . \]
(Note that the fiber product of \cref{define 0th coskel and pb of simp obj} that yields this pullback may be equivalently taken either in $\SsS$ or in $s(s\S)$, in light of the left localization adjunction of \cref{define Segal sspaces}.)  This clearly assembles to a functor, and in fact it is not hard to see that this participates in a right localization adjunction
\[ \begin{tikzcd}[column sep=2cm]
\CatsS \arrow[hook, transform canvas={yshift=0.7ex}]{r}{\forget_\CatsS}[swap, transform canvas={yshift=0.2ex}]{\scriptstyle \bot} \arrow[leftarrow, transform canvas={yshift=-0.7ex}]{r}[swap]{\spat}
& \SsS ,
\end{tikzcd} \]
whose counit components $\spat(\C_\bullet) \ra \C_\bullet$ are Dwyer--Kan weak equivalences (which are even fully faithful as in \cref{any SsS is DK-equivt to a sS-cat}).
\end{defn}

\subsection{$\SsS$ and $\CatsS$ as presentations of $\Cati$}\label{subsection SsS and Cat_sS present infty-cats}

The following pair of results asserts that both $s\S$-enriched $\infty$-categories and Segal simplicial spaces, equipped with their respective subcategories of Dwyer--Kan weak equivalences, present the $\infty$-category of $\infty$-categories.

\begin{prop}\label{sS-enriched cats model infty-cats}
The composite functor
\[ \CatsS \xra{|{-}|} \SS \xra{\leftloc_\CSS} \CSS \simeq \Cati \]
induces an equivalence
\[ \loc{\CatsS}{\bW_\DK} \xra{\sim} \CSS \simeq \Cati . \]
\end{prop}

\begin{proof}
So far, we have obtained the solid diagram
\[ \begin{tikzcd}
s(s\S)
{\arrow[transform canvas={yshift=0.7ex}]{rr}{s(|{-}|)}[swap, transform canvas={yshift=0.25ex}]{\scriptstyle \bot} \arrow[transform canvas={yshift=-0.7ex}, hookleftarrow]{rr}[swap]{s(\const)}}
{\arrow[transform canvas={xshift=0.5ex, yshift=0.5ex}]{rd}[swap, sloped, pos=0.65, transform canvas={yshift=0.3ex}]{\scriptstyle \bot} \arrow[transform canvas={xshift=-0.5ex, yshift=-0.5ex}, hookleftarrow]{rd}}
& & s\S
{\arrow[transform canvas={xshift=0.5ex, yshift=0.5ex}]{rd}[sloped, pos=0.25]{\leftloc_\SS}[swap, sloped, pos=0.65, transform canvas={yshift=0.3ex}]{\scriptstyle \bot} \arrow[transform canvas={xshift=-0.5ex, yshift=-0.5ex}, hookleftarrow]{rd}[swap, sloped, pos=0.8]{\forget_\SS}}
\\
& \SsS & & \SS
{\arrow[transform canvas={xshift=0.5ex, yshift=0.5ex}]{rd}[sloped, pos=0.15]{\leftloc_\CSS}[swap, sloped, pos=0.65, transform canvas={yshift=0.3ex}]{\scriptstyle \bot} \arrow[transform canvas={xshift=-0.5ex, yshift=-0.5ex}, hookleftarrow]{rd}[swap, sloped, pos=0.85]{\forget_\CSS}}
\\
\CatsS
{\arrow[transform canvas={xshift=-0.5ex, yshift=0.5ex}, hookrightarrow]{ru}[swap, sloped, pos=0.35, transform canvas={yshift=0.3ex}]{\scriptstyle \bot} \arrow[transform canvas={xshift=0.5ex, yshift=-0.5ex}, leftarrow]{ru}[sloped, swap, pos=0.3]{\spat}}
{\arrow[transform canvas={xshift=-0.2ex, yshift=0.8ex}, bend right=20]{rrru}[sloped, pos=0.57]{|{-}|}[swap, sloped, pos=0.45, transform canvas={yshift=0.2ex}]{\scriptstyle \bot} \arrow[transform canvas={xshift=0.2ex, yshift=-0.8ex}, bend right=20, hookleftarrow, dashed]{rrru}[swap, sloped, pos=0.35]{s(\const)}}
& & & & \CSS .
\end{tikzcd} \]
The right adjoint of the composite left localization adjunction
\[ \begin{tikzcd}[column sep=1.5cm]
s(s\S) 
{\arrow[transform canvas={yshift=0.7ex}]{r}{s(|{-}|)}[swap, transform canvas={yshift=0.25ex}]{\scriptstyle \bot} \arrow[transform canvas={yshift=-0.7ex}, hookleftarrow]{r}[swap]{s(\const)}}
& s\S
{\arrow[transform canvas={yshift=0.7ex}]{r}{\leftloc_\SS}[swap, transform canvas={yshift=0.25ex}]{\scriptstyle \bot} \arrow[transform canvas={yshift=-0.7ex}, hookleftarrow]{r}[swap]{\forget_\SS}}
& \SS
\end{tikzcd} \]
clearly lands in the full subcategory $\CatsS \subset s(s\S)$, and hence restricts to give the right adjoint of a left localization adjunction as indicated by the dotted arrow above.  This composes to a left localization adjunction
\[ \begin{tikzcd}[column sep=1.5cm]
\CatsS
{\arrow[transform canvas={yshift=0.7ex}]{r}{|{-}|}[swap, transform canvas={yshift=0.25ex}]{\scriptstyle \bot} \arrow[transform canvas={yshift=-0.7ex}, hookleftarrow]{r}[swap]{s(\const)}}
& \SS
{\arrow[transform canvas={yshift=0.7ex}]{r}{\leftloc_\CSS}[swap, transform canvas={yshift=0.25ex}]{\scriptstyle \bot} \arrow[transform canvas={yshift=-0.7ex}, hookleftarrow]{r}[swap]{\forget_\CSS}}
& \CSS .
\end{tikzcd} \]
Moreover, the definition of Dwyer--Kan weak equivalence is precisely chosen so that the composite left adjoint creates the subcategory $\bW_\DK \subset \CatsS$.  Hence, by \cref{rnerves:ex left localization qua free localization}, it does indeed induce an equivalence
\[ \loc{\CatsS}{\bW_\DK} \xra{\sim} \CSS \simeq \Cati , \]
as desired.
\end{proof}


\begin{prop}\label{spatializn adjn induces equivce on localizns}
Both adjoints in the right localization adjunction
\[ \begin{tikzcd}[column sep=2cm]
\CatsS
{\arrow[transform canvas={yshift=0.7ex}, hookrightarrow]{r}{\forget_\CatsS}[swap, transform canvas={yshift=0.25ex}]{\scriptstyle \bot} \arrow[transform canvas={yshift=-0.7ex}, leftarrow]{r}[swap]{\spat}}
& \SsS
\end{tikzcd} \]
are functors of relative $\infty$-categories (with respect to their respective Dwyer--Kan relative structures), and moreover they induce inverse equivalences
\[ \loc{\CatsS}{(\bW^{\CatsS}_\DK)} \simeq \loc{\SsS}{(\bW^\SsS_\DK)} \]
in $\Cati$ on localizations.
\end{prop}

\begin{proof}
The left adjoint inclusion is a functor of relative $\infty$-categories by definition.  On the other hand, suppose that $\C_\bullet \we \D_\bullet$ is a map in $\bW^\SsS_\DK \subset \SsS$.  Via the right localization adjunction, its spatialization fits into a commutative diagram
\[ \begin{tikzcd}
\spat(\C_\bullet) \arrow{r}{\approx} \arrow{d} & \C_\bullet \arrow{d}[sloped, anchor=south]{\approx} \\
\spat(\D_\bullet) \arrow{r}[swap]{\approx} & \D_\bullet
\end{tikzcd} \]
in $\SsS_\DK$, and hence is also in $\bW^\SsS_\DK \subset \SsS$ by the two-out-of-three property.  This shows that the right adjoint is also a functor of relative $\infty$-categories.

To see that these adjoints induce inverse equivalences on localizations, note that the composite
\[ \CatsS \xhookra{\forget_\CatsS} \SsS \xra{\spat} \CatsS \]
is the identity, while the composite
\[ \SsS \xra{\spat} \CatsS \xhookra{\forget_\CatsS} \SsS \]
admits a natural weak equivalence in $\SsS_\DK$ to the identity functor (namely, the counit of the adjunction).  Hence, the claim follows from \cref{rnerves:nat w.e. induces equivce betw fctrs}.
\end{proof}


To conclude this section, we make a pair of general remarks regarding $\SsS$ and $\CatsS$.  We begin by contextualizing these $\infty$-categories with respect to Barwick's theory of enriched $\infty$-categories, which is described in \cite[\sec 1]{LurieGoo}.

\begin{rem}\label{compare with barwick's theory of enriched infty-cats}
Barwick's theory of enriched $\infty$-categories -- which provides a satisfactory, compelling, and apparently complete picture (at least when the enriching $\infty$-category is equipped with the \textit{cartesian} symmetric monoidal structure) -- is premised on the notion of a \textit{distributor}, the data of which is simply an $\infty$-category $\Y$ equipped with a full subcategory $\X \subset \Y$ (see \cite[Definition 1.2.1]{LurieGoo}).\footnote{Note that there is a typo in \cite[Definition 1.2.1]{LurieGoo}: condition (4) should say that the functor $\X \ra (\Cati)^{op}$ preserves \textit{colimits}, not limits.  This is clear from \cite[Example 1.2.3]{LurieGoo} (see Lemma T.6.1.3.7 and Definition T.6.1.3.8).}  Given such a distributor, one can then define $\infty$-categories $\SS_{\X \subset \Y}$ and $\CSS_{\X \subset \Y}$ of \textit{Segal space objects} and of \textit{complete Segal space objects} with respect to it: these sit as full (in fact, reflective) subcategories
\[ \CSS_{\X \subset \Y} \subset \SS_{\X \subset \Y} \subset s\Y , \]
in which
\begin{itemizesmall}
\item the subcategory $\SS_{\X \subset \Y} \subset s\Y$ consists of those simplicial objects $Y_\bullet \in s\Y$ such that
\begin{itemizesmall}
\item $Y_\bullet$ satisfies the Segal condition and
\item $Y_0 \in \X$
\end{itemizesmall}
(see \cite[Definition 1.2.7]{LurieGoo}), while
\item the subcategory $\CSS_{\X \subset \Y} \subset \SS_{\X \subset \Y}$ consists of those objects which additionally satisfy a certain \textit{completeness} condition (see \cite[Definition 1.2.10]{LurieGoo}).
\end{itemizesmall}
Thus, $\Y$ plays the role of the ``enriching $\infty$-category'', i.e.\! the $\infty$-category containing the hom-objects in our enriched $\infty$-category, while its subcategory $\X \subset \Y$ provides a home for the ``object of objects'' of the enriched $\infty$-category.  As in the classical case -- indeed, the identity distributor $\S \subset \S$ simply has $\SS_{\S \subset \S} \simeq \SS$ and $\CSS_{\S \subset \S} \simeq \CSS$ --, one can already meaningfully extract an enriched $\infty$-category from a Segal space object, but it is only by restricting to the complete ones that one obtains the desired $\infty$-category of such.

Now, obviously we have
\[ \SsS \simeq \SS_{s\S \subset s\S} ,\]
as Segal simplicial spaces are nothing but Segal space objects with respect to the identity distributor $s\S \subset s\S$ on the $\infty$-category $s\S$ of simplicial spaces.  We can clearly also identify the $\infty$-category of $s\S$-enriched $\infty$-categories as
\[ \CatsS \simeq \SS_{\S \subset s\S} , \]
the Segal space objects with respect to the distributor $\S \subset s\S$ (the embedding of spaces as the constant simplicial spaces).\footnote{To see that the inclusion $\S \subset s\S$ of the full subcategory of constant objects is a distributor, note that if $\Y$ is an $\infty$-topos and $\X \subset \Y$ is a full subcategory which is stable under limits and colimits, then $\X \subset \Y$ is automatically a distributor.  The only remaining point is to verify condition (4) of \cite[Definition 1.2.1]{LurieGoo}.  The functor $\X \ra (\Cati)^{op}$ is given on objects by $x \mapsto (\Y_{/x})^\opobj$, with functoriality given by pullback in $\Y$.  This clearly factors as the composite $\X \hookra Y \ra (\Cati)^{op}$, in which the latter functor is similarly given by $y \mapsto (\Y_{/y})^\opobj$, which then preserves colimits by Proposition T.6.1.3.10 and Theorem T.6.1.3.9.}\footnote{In contrast with \cref{rem SsS as simplicial cats}, $s\S$-enriched $\infty$-categories do not quite have an analog in ordinary category theory, only in enriched category theory.  (It is only a coincidence of the special case presently under study that the two $\infty$-categories $\S$ and $s\S$ participating in the distributor appear to be so closely related.)}  On the other hand, the subcategory
\[ \CSS_{\S \subset s\S} \subset \SS_{\S \subset s\S} \simeq \CatsS \]
consists of those $s\S$-enriched $\infty$-categories $\C_\bullet \in \CatsS$ such that the ``levelwise $0\th$ space'' object $(\C_\bullet)_0 \in s\S$ is constant.
\end{rem}

We now explain the source of our interest in the $\infty$-categories $\SsS$ and $\CatsS$.

\begin{rem}\label{sS-enr infty-cats are segal space objects etc}
First and foremost, the reason we are interested in $\SsS$ is because this is the natural target of the ``pre-hammock localization'' functor
\[ \RelCati \xra{\preham} \SsS , \]
whose construction constitutes the main ingredient of the construction of the hammock localization functor itself (see \cref{section hammocks}).  On the other hand, we then restrict to the (coreflective) subcategory $\CatsS \subset \SsS$ since this appears to be the largest full subcategory of $\SsS \subset s(s\S)$ on which the levelwise geometric realization functor
\[ s(s\S) \xra{s(|{-}|)} s\S \]
(which is a colimit) preserves the Segal condition (which is defined in terms of limits), at least for purely formal reasons (recall (the proof of) \cref{lw geom realizn of sS-enr infty-cats gives segal spaces}).  Indeed, if our ``local geometric realization'' functor failed to preserve the Segal condition, it would necessarily destroy all ``category-ness'' inherent in our objects of study.  In turn, this would effectively invalidate our right to declare the hammock simplicial spaces
\[ \homhamR(x,y) \in s\S \]
(see \cref{define hammocks}) -- which will of course be the hom-simplicial spaces in the hammock localization $\ham(\R,\bW) \in \CatsS$ -- as ``presentations of hom-spaces'' in any reasonable sense.

For these reasons, Segal simplicial spaces are therefore not really our primary interest.  However, since for a Segal simplicial space $\C_\bullet \in \SsS$, the counit $\spat(\C_\bullet) \ra \C_\bullet$ of the spatialization right localization adjunction is actually fully faithful in the $s\S$-enriched sense, the hammock localization
\[ \ham(\R,\bW) = \spat ( \preham(\R,\bW)) \in \CatsS \]
will then simultaneously
\begin{itemizesmall}
\item have the hammock simplicial spaces as its hom-simplicial spaces, and
\item have composition maps which both
\begin{itemizesmall}
\item directly present composition in its geometric realization, and
\item manifestly encode the notion of ``concatenation of zigzags''.
\end{itemizesmall}
\end{itemizesmall}

Of course, it would also be possible to restrict further to the (reflective) subcategory
\[ \CSS_{\S \subset s\S} \subset \SS_{\S \subset s\S} \simeq \CatsS \]
of complete Segal space objects (recall \cref{compare with barwick's theory of enriched infty-cats}).  However, this is unnecessary for our purposes, since we have already proved that both the pre-hammock localization functor and the hammock localization functor
land in $\infty$-categories which admit canonical (Dwyer--Kan) relative structures via which they present the $\infty$-category $\Cati$, thus endowing these constructions with external meaning (which are of course compatible with each other in light of \cref{spatializn adjn induces equivce on localizns}).  Moreover, as the successive inclusions
\[ \CSS_{\S \subset s\S} \subset \SS_{\S \subset s\S} \simeq \CatsS \subset \SsS \]
respectively admit a \textit{left} adjoint and a \textit{right} adjoint, this further restriction would in all probability make for a somewhat messier story.
\end{rem}

\section{Zigzags and hammocks in relative $\infty$-categories}\label{section zigzags and hammocks}

In studying relative 1-categories and their 1-categorical localizations, one is naturally led to study \textit{zigzags}.  Given a relative category $(\R,\bW) \in \RelCat$ and a pair of objects $x,y \in \R$, a zigzag from $x$ to $y$ is a diagram of the form
\[ x \lwe \cdots \ra \cdots \lwe \cdots \ra \cdots \lwe y , \]
i.e.\! a sequence of both forwards and backwards morphisms in $\R$ (in arbitrary (finite) quantities and in any order) such that all backwards morphisms lie in $\bW \subset \R$.  Under the localization $\R \ra \R[\bW^{-1}]$, such a diagram is taken to a sequence of morphisms such that all backwards maps are \textit{isomorphisms}, so that it is in effect just a sequence of composable (forwards) arrows.  Taking their composite, we obtain a single morphism $x \ra y$ in the 1-categorical localization $\R[\bW^{-1}]$.  In fact, one can explicitly construct $\R[\bW^{-1}]$ in such a way that \textit{all} of its morphisms arise from this procedure.

It is a good deal more subtle to show, but in fact the same is true of relative \textit{$\infty$-categories} and their ($\infty$-categorical) localizations: given a relative $\infty$-category $(\R,\bW) \in \RelCati$, it turns out that every morphism in $\loc{\R}{\bW}$ can likewise be presented by a zigzag in $(\R,\bW)$ itself.  (We prove a precise statement of this assertion as \cref{rep maps in localizn of rel infty-cat by zigzags}.)

The representation of a morphism in $\loc{\R}{\bW}$ by a zigzag in $(\R,\bW)$ is quite clearly overkill: many different zigzags in $(\R,\bW)$ will present the same morphism in $\loc{\R}{\bW}$.  For example, we can consider a zigzag as being selected by a morphism $\word{m} \ra (\R,\bW)$ of relative $\infty$-categories, where $\word{m} \in \RelCat \subset \RelCati$ is a \textit{zigzag type} which is determined by the shape of the zigzag in question; then, precomposition with a suitable morphism $\word{m}' \ra \word{m}$ of zigzag types will yield a composite $\word{m}' \ra \word{m} \ra (\R,\bW)$ which presents a canonically equivalent morphism in $\loc{\R}{\bW}$.  Thus, in order to obtain a closer approximation to $\hom_{\loc{\R}{\bW}}(x,y)$, we should take a \textit{colimit} of the various spaces of zigzags from $x$ to $y$ indexed over the category of zigzag types.

However, this colimit alone will still not generally capture all the redundancy inherent in the representation of morphisms in $\loc{\R}{\bW}$ by zigzags in $(\R,\bW)$.  Namely, a \textit{natural weak equivalence} between two zigzags of the same type (which fixes the endpoints) will, upon postcomposing to the localization $\R \ra \loc{\R}{\bW}$, yield a \textit{homotopy} between the morphisms presented by the respective zigzags.  Pursuing this observation, we are thus led to consider certain \textit{$\infty$-categories}, denoted $\word{m}(x,y)$ (for varying zigzag types $\word{m}$), whose objects are the $\word{m}$-shaped zigzags from $x$ to $y$ and whose morphisms are the natural weak equivalences (fixing $x$ and $y$) between them.

Finally, putting these two observations of redundancy together, we see that in order to approximate the hom-space $\hom_{\loc{\R}{\bW}}(x,y)$, we should be taking a colimit of the various $\infty$-categories $\word{m}(x,y)$ over the category of zigzag types.  In fact, rather than taking a colimit of these $\infty$-categories, we will take a colimit of their corresponding \textit{complete Segal spaces} (see \cref{rnerves:section CSSs}), not within the $\infty$-category $\CSS$ of such but rather within the larger ambient $\infty$-category $s\S$ in which it is definitionally contained; this, finally, will yield the \textit{hammock simplicial space} $\homhamR(x,y) \in s\S$, which (as the notation suggests) will be the hom-simplicial space in the hammock localization $\ham(\R,\bW) \in \CatsS$.\footnote{As the functor $\leftloc_\CSS : s\S \ra \CSS$ is left adjoint to the inclusion $\CSS \subset s\S$ and hence in particular commutes with colimits, its application to the hammock simplicial space will yield the aforementioned colimit of $\infty$-categories.  Moreover, since we are ultimately interested in hammock simplicial spaces for their geometric realizations, in view of \cref{rnerves:groupoid-completion of CSSs} we can consider this shift in ambient $\infty$-category merely as a technical convenience.  For instance, there is an evident explicit description of the constituent spaces in the hammock simplicial space (analogous to the 1-categorical case (see \cite[2.1]{DKCalc})).}

This section is organized as follows.
\begin{itemize}

\item In \cref{subsection doubly-ptd rel infty-cats}, we lay some groundwork regarding \textit{doubly-pointed} relative $\infty$-categories, which will allow us to efficiently corepresent our $\infty$-categories of zigzags.

\item In \cref{subsection zigzags}, we use this to define $\infty$-categories of zigzags in a relative $\infty$-category.

\item In \cref{subsection representing maps}, we prove a precise articulation of the assertion made above, that all morphisms in the localization $\loc{\R}{\bW}$ are represented by zigzags in $(\R,\bW)$.

\item In \cref{subsection hammock sspaces}, we finally define our hammock simplicial spaces and compare them with the hammock simplicial sets of Dwyer--Kan (in the special case of a relative 1-category).

\item In \cref{subsection gluing zigzags}, we assemble some technical results regarding zigzags in relative $\infty$-categories which will be useful later; notably, we prove that for a concatenation $[\word{m} ; \word{m}']$ of zigzag types, we can recover the $\infty$-category $[\word{m};\word{m}'] (x,y)$ via the \textit{two-sided Grothendieck construction} (see \cref{gr:define two-sided Gr}).

\end{itemize}

\subsection{Doubly-pointed relative $\infty$-categories}\label{subsection doubly-ptd rel infty-cats}

In this subsection, we make a number of auxiliary definitions which will streamline our discussion throughout the remainder of this paper.

\begin{defn}\label{define doubly-pointed relcats}
A \bit{doubly-pointed relative $\infty$-category} is a relative $\infty$-category $(\R,\bW)$ equipped with a map $\pt_\RelCati \sqcup \pt_\RelCati \ra \R$.  The two inclusions $\pt_\RelCati \hookra \pt_\RelCati \sqcup \pt_\RelCati$ select objects $s,t \in \R$, which we call the \bit{source} and the \bit{target}; we will sometimes subscript these to remove ambiguity, e.g.\! as $s_\R$ and $t_\R$.  These assemble into the evident $\infty$-category, which we denote by
\[ \RelCatip = (\RelCati)_{(\pt_\RelCati \sqcup \pt_\RelCati) /} . \]
Of course, there is a forgetful functor $\RelCatip \ra \RelCati$.  We will often implicitly consider a relative $\infty$-category $(\R,\bW)$ equipped with two chosen objects $x,y \in \R$ as a doubly-pointed relative $\infty$-category; on the other hand, we may also write $((\R,\bW),x,y) \in \RelCatip$ to be more explicit.  We write $\RelCatp \subset \RelCatip$ for the full subcategory of \bit{doubly-pointed relative categories}, i.e.\! of those doubly-pointed relative $\infty$-categories whose underlying $\infty$-category is a 1-category.
\end{defn}

\begin{notn}\label{define enrichment and tensoring of doubly-pointed relcats over relcats}
Recall from \cref{rnerves:define internal hom in relcats} that $\RelCati$ is a cartesian closed symmetric monoidal $\infty$-category.  With respect to this structure, $\RelCatip$ is enriched and tensored over $\RelCati$.  As for the enrichment, for any $(\R_1\bW_1) , (\R_2,\bW_2) \in \RelCatip$, we define the object
\[ \left( \Funp(\R_1,\R_2)^\Rel , \Funp(\R_1,\R_2)^\bW \right) = \lim \left( \begin{tikzcd}
 & \left( \Fun(\R_1,\R_2)^\Rel , \Fun(\R_1,\R_2)^\bW \right) \arrow{d}{(\ev_{s_1} , \ev_{t_1})} \\
\{ (s_2,t_2) \} \arrow[hook]{r} & (\R_2,\bW_2) \times (\R_2,\bW_2)
\end{tikzcd} \right) \]
of $\RelCati$ (where we write $s_1,t_1 \in \R_1$ and $s_2,t_2 \in \R_2$ to distinguish between the source and target objects); informally, this should be thought of as the relative $\infty$-category whose objects are the doubly-pointed relative functors from $(\R_1,\bW_1)$ to $(\R_2,\bW_2)$, whose morphisms are the doubly-pointed natural transformations between these (i.e.\! those natural transformations whose components at $s_1$ and $t_1$ are $\id_{s_2}$ and $\id_{t_2}$, resp.), and whose weak equivalences are the doubly-pointed natural weak equivalences.  Then, the tensoring is obtained by taking $(\R,\bW) \in \RelCati$ and $(\R_1,\bW_1) \in \RelCatip$ to the pushout
\[ \colim \left( \begin{tikzcd}
\R \times \{s,t\} \arrow{r} \arrow{d} & \R \times \R_1 \\
\pt_\RelCati \times \{s,t\}
\end{tikzcd} \right) \]
in $\RelCati$, with its double-pointing given by the natural map from $\pt_\RelCati \sqcup \pt_\RelCati \simeq \pt_\RelCati \times \{ s, t\}$.  We will write
\[ \RelCatip \times \RelCati \xra{- \tensoring -} \RelCatip \]
to denote this tensoring.
\end{notn}

\begin{notn}\label{define maybe-pointed relcats}
In order to simultaneously refer to the situations of unpointed and doubly-pointed relative $\infty$-categories, we will use the notation $\RelCatimp$ (and similarly for other related notations).  When we use this notation, we will mean for the entire statement to be interpreted either in the unpointed context or the doubly-pointed context.
\end{notn}

\begin{notn}\label{notn for tensoring of either relcat or relcatp over relcat}
We will write
\[ \RelCatimp \times \RelCati \xra{- \tensoring -} \RelCatimp \]
to denote either the tensoring of \cref{define enrichment and tensoring of doubly-pointed relcats over relcats} in the doubly-pointed case or else simply the cartesian product in the unpointed case.
\end{notn}

\subsection{Zigzags in relative $\infty$-categories}\label{subsection zigzags}

In this subsection we introduce the first of the two key concepts of this section, namely the \textit{$\infty$-categories of zigzags} in a relative $\infty$-category between two given objects.

We begin by defining the objects which will corepresent our $\infty$-categories of zigzags.

\begin{defn}\label{define relative word}
We define a \bit{relative word} to be a (possibly empty) word $\word{m}$ in the symbols $\any$ (for ``any arbitrary arrow'') and $\bW^{-1}$.  We will write $\any^{\circ n}$ to denote $n$ consecutive copies of the symbol $\any$ (for any $n \geq 0$), and similarly for $(\bW^{-1})^{\circ n}$.  We can extract a doubly-pointed relative category from a relative word, which for our sanity we will carry out by reading \textit{forwards}. So for instance, the relative word $\word{m} = [ \any ; (\bW^{-1})^{\circ 2} ; \any^{\circ 2} ]$ defines the doubly-pointed relative category
\[ \begin{tikzcd}
s \arrow{r} & \bullet & \bullet \arrow{l}[swap]{\approx} & \bullet \arrow{l}[swap]{\approx} \arrow{r} & \bullet \arrow{r} & t.
\end{tikzcd} \]
We denote this object by $\word{m} \in \RelCatp$.  Thus, by convention, the empty relative word determines the terminal object $[\es] \simeq \pt_{\RelCatp} \in \RelCatp$ (which is the unique relative word determining a doubly-pointed relative category whose source and target objects are equivalent).  Restricting to the \textit{order-preserving} maps between relative words (with respect to the evident ordering on their objects, i.e.\! starting from $s$ and ending at $t$), we obtain a (non-full) subcategory $\Z \subset \RelCatp$ of \bit{zigzag types}.\footnote{Note that the objects of $\Z$ can in fact be considered as \textit{strict} doubly-pointed relative categories, and moreover $\Z$ itself can be considered as a \textit{strict} category.  However, as we will only use these objects in invariant manipulations, we will not need these observations.}\footnote{Omitting the terminal relative word from $\Z$ (and considering it as a strict category), we obtain the opposite of the indexing category $\II$ of \cite[4.1]{DKCalc}.  We prefer to include this terminal object: it is the unit object for a monoidal structure on $\Z$ given by concatenation, which will play a key role in the definition of the hammock localization (see \cref{main constrn for hammock localizn}).}\footnote{Note that an order-preserving map must lay each morphism $[\any]$ across some $[\any^{\circ m}]$ (for some $m \geq 0$), and must lay each morphism $[\bW^{-1}]$ across some $[(\bW^{-1})^{\circ n}]$ (for some $n \geq 0$).  In particular, it cannot lay a morphism $[\any]$ across a morphism $[\bW^{-1}]$ (or vice versa, of course).}  We will occasionally also use this same relative word notation with the symbol $\bW$, but the resulting doubly-pointed relative categories will not be objects of $\Z$.
\end{defn}

\begin{rem}\label{rel word concatenation is a pushout}
Let $\word{m},\word{m}' \in \Z \subset \RelCatp \subset \RelCatip$ be relative words.  Then, their concatenation can be characterized as a pushout
\[ \begin{tikzcd}
\pt_\RelCati \arrow{r}{s} \arrow{d}[swap]{t} & \word{m}' \arrow{d} \\
\word{m} \arrow{r} & {[\word{m};\word{m}']}
\end{tikzcd} \]
in $\RelCati$ (as well as in $\RelCat$).
\end{rem}

\begin{notn}
For any $\word{m} \in \Z$, we will write $|\word{m}|_\any \in \bbN$ to denote the number of times that $\any$ appears in $\word{m}$, and we will write $|\word{m}|_{\bW^{-1}} \in \bbN$ to denote the number of times that $\bW^{-1}$ appears in $\word{m}$.
\end{notn}

\begin{rem}
The localization functor
\[ \RelCati \xra{\locL} \Cati \]
acts on the subcategory $\Z \subset \RelCat \subset \RelCati$ of zigzag types as
\[ \locL(\word{m}) \simeq [ |\word{m}|_\any ] \in \bD \subset \Cat \subset \Cati : \]
in effect, it collapses all the copies of $[\bW^{-1}]$ and leaves the copies of $[\any]$ untouched.
\end{rem}

We now define the first of the two key concepts of this section, an analog of \cite[5.1]{DKCalc}.

\begin{defn}\label{define zigzags}
Given a relative $\infty$-category $(\R,\bW)$ equipped with two chosen objects $x,y \in \R$, and given a relative word $\word{m} \in \Z$, we define the $\infty$-category of \bit{zigzags} in $(\R,\bW)$ from $x$ to $y$ of type $\word{m}$ to be
\[ \word{m}_{(\R,\bW)}(x,y) = \Funp(\word{m},\R)^\bW . \]
If the relative $\infty$-category $(\R,\bW)$ is clear from context, we will simply write $\word{m}(x,y)$.
\end{defn}

\subsection{Representing maps in $\loc{\R}{\bW}$ by zigzags in $(\R,\bW)$}\label{subsection representing maps}

In this subsection, we take a digression to illustrate that our study of zigzags in relative $\infty$-categories is well-founded: roughly speaking, we show that any morphism in the localization of a relative $\infty$-category is represented by a zigzag in the relative $\infty$-category itself.  We will give the precise assertion as \cref{rep maps in localizn of rel infty-cat by zigzags}.  In order to state it, however, we first introduce the following terminology.

\begin{defn}\label{defn rep morphism in localizn by zigzag}
Let $(\R,\bW_\R)$ and $(\D,\bW_\D)$ be relative $\infty$-categories.  We will say that a morphism
\[ (\D,\bW_\D) \ra (\R,\bW_\R) \]
in $\RelCati$ \bit{represents} the morphism
\[ \loc{\D}{\bW_\D} \ra \loc{\R}{\bW_\R} \]
in $\Cati$ induced by the localization functor.  We will also say that it represents the morphism
\[ \ho(\loc{\D}{\bW_\D}) \ra \ho(\loc{\R}{\bW_\R}) \]
in $\Cat$ induced from the previous one by the homotopy category functor.  In a slight abuse of terminology, we will moreover say that a zigzag
\[ \word{m} \ra (\R,\bW_\R) \]
represents the composite
\[ [1] \ra \locL(\word{m}) \ra \loc{\R}{\bW_\R} \]
in $\Cati$, where the map $[1] \ra \locL(\word{m}) \simeq [ | \word{m}|_\any ]$ is given by $0 \mapsto 0$ and $1 \mapsto | \word{m}|_\any$ (i.e.\! it corepresents the operation of composition), and likewise for the morphism in the homotopy category $\ho(\loc{\R}{\bW_\R})$ of the localization selected by either three-fold composite in the commutative diagram
\[ \begin{tikzcd}[column sep=0.75cm]
{[1]} \arrow{rd} \\
& \locL(\word{m}) \arrow{r} \arrow{dd}[sloped, anchor=north]{\sim} & \loc{\R}{\bW_\R} \arrow{dd} \\
\ \\
& \ho(\locL(\word{m})) \arrow{r} & \ho(\loc{\R}{\bW_\R})
\end{tikzcd} \]
in $\Cati$.
\end{defn}

\begin{prop}\label{rep maps in localizn of rel infty-cat by zigzags}
Let $(\R,\bW) \in \RelCati$ be a relative $\infty$-category, and let $[1] \xra{F} \loc{\R}{\bW}$ be a functor selecting a morphism in its localization.  Then, for some relative word $\word{m} \in \Z$, there exists a zigzag $\word{m} \ra (\R,\bW)$ which represents $F$.
\end{prop}

We will prove \cref{rep maps in localizn of rel infty-cat by zigzags} in stages of increasing generality.  We begin by recalling that any morphism in the \textit{1-categorical} localization of a relative \textit{1-category} is represented by a zigzag.

\begin{lem}\label{rep maps in 1-cat localizn of rel 1-cat by zigzags}
Let $(\R,\bW) \in \RelCat$ be a relative 1-category, and let $[1] \xra{F} \R[\bW^{-1}]$ be a functor selecting a morphism in its 1-categorical localization.  Then, for some relative word $\word{m} \in \Z$, there exists a zigzag $\word{m} \ra (\R,\bW)$ which represents $F$.
\end{lem}

\begin{proof}
This follows directly from the standard construction of the 1-categorical localization of a relative 1-category.
\end{proof}

\begin{rem}
\cref{rep maps in 1-cat localizn of rel 1-cat by zigzags} accounts for the fundamental role that zigzags play in the theory of relative categories and their 1-categorical localizations.  We can therefore view \cref{rep maps in localizn of rel infty-cat by zigzags} as asserting that zigzags play an analogous fundamental role in the theory of relative $\infty$-categories and their ($\infty$-categorical) localizations.
\end{rem}

\begin{rem}\label{diagram of rep maps in 1-cat localizn of rel 1-cat by zigzags}
We can view \cref{rep maps in 1-cat localizn of rel 1-cat by zigzags} as guaranteeing the existence of a diagram
\[ \begin{tikzcd}
& \word{m} \arrow[dashed]{r} \arrow[maps to]{d} & (\R,\bW) \arrow[maps to]{d} \\
& \ho(\locL(\word{m})) \arrow[dashed]{r} & {\R[\bW^{-1}]} \\
{[1]} \arrow{ru} \arrow{rru}[swap, sloped]{F}
\end{tikzcd} \]
for some relative word $\word{m} \in \Z$, in which
\begin{itemizesmall}
\item the upper dotted arrow is a morphism in $\RelCat \subset \RelCati$,
\item the lower dotted arrow is its image under the 1-categorical localization functor
\[ \RelCati \xra{\locL} \Cati \xra{\ho} \Cat , \]
and
\item the map $[1] \ra \ho(\locL(\word{m})) \simeq \ho([|\word{m}|_\any]) \simeq [ |\word{m}|_\any]$ is as in \cref{defn rep morphism in localizn by zigzag}.
\end{itemizesmall}
\end{rem}

With \cref{rep maps in 1-cat localizn of rel 1-cat by zigzags} recalled, we now move on to the case of \textit{$\infty$-categorical} localizations of relative 1-categories.

\begin{lem}\label{rep maps in infty-cat localizn of rel 1-cat by zigzags}
Let $(\R,\bW) \in \RelCat$ be a relative 1-category, and let $[1] \xra{F} \loc{\R}{\bW}$ be a functor selecting a morphism in its localization.  Then, for some relative word $\word{m} \in \Z$, there exists a zigzag $\word{m} \ra (\R,\bW)$ which represents $F$.
\end{lem}

\begin{proof}
Recall from \cref{rnerves:loc and ho commute} that we have an equivalence $\ho(\loc{\R}{\bW}) \xra{\sim} \R[\bW^{-1}]$.  The resulting postcomposition
\[ [1] \xra{F} \loc{\R}{\bW} \ra \ho(\loc{\R}{\bW}) \xra{\sim} \R[\bW^{-1}] \]
of $F$ with the projection to the homotopy category selects a morphism in the 1-categorical localization $\R[\bW^{-1}]$.  Hence, by \cref{rep maps in 1-cat localizn of rel 1-cat by zigzags}, we obtain a diagram
\[ \begin{tikzcd}
& \word{m} \arrow{r} \arrow[maps to]{d} & (\R,\bW) \arrow[maps to]{d} \\
& \locL(\word{m}) \arrow[dashed]{r} \arrow{d}[sloped, anchor=north]{\sim} & \loc{\R}{\bW} \arrow{d} \\
& \ho(\locL(\word{m})) \arrow{r} & {\R[\bW^{-1}]} \\
{[1]} \arrow{ru} \arrow{rru} \arrow{ruu}
\end{tikzcd} \]
for some relative word $\word{m} \in \Z$, in which
\begin{itemizesmall}
\item the solid horizontal arrows are as in \cref{diagram of rep maps in 1-cat localizn of rel 1-cat by zigzags},
\item the upper map in $\RelCat \subset \RelCati$ induces the dotted map under the functor $\locL : \RelCati \ra \Cati$, so that
\item the (lower) square in $\Cati$ commutes.
\end{itemizesmall}
That the resulting composite
\[ [1] \ra \locL(\word{m}) \ra \loc{\R}{\bW} \]
is equivalent to the functor $[1] \xra{F} \loc{\R}{\bW}$ follows from \cref{space of lifts of a map in htpy cat is connected}.  Thus, in effect, we obtain a diagram
\[ \begin{tikzcd}
& \word{m} \arrow[dashed]{r} \arrow[maps to]{d} & (\R,\bW) \arrow[maps to]{d} \\
& \locL(\word{m}) \arrow[dashed]{r} & \loc{\R}{\bW} \\
{[1]} \arrow{ru} \arrow{rru}
\end{tikzcd} \]
analogous to the one in \cref{diagram of rep maps in 1-cat localizn of rel 1-cat by zigzags} (only with the 1-categorical localizations replaced by the $\infty$-categorical localizations), which proves the claim.
\end{proof}

\begin{lem}\label{space of lifts of a map in htpy cat is connected}
For any $\infty$-category $\C$ and any map $[1] \ra \ho(\C)$, the space of lifts
\[ \begin{tikzcd}
& \C \arrow{d} \\
{[1]} \arrow{r} \arrow[dashed]{ru} & \ho(\C)
\end{tikzcd} \]
is connected.
\end{lem}

\begin{proof}
Since the functor $\C \ra \ho(\C)$ creates the subcategory $\C^\simeq \subset \C$, there is a connected space of lifts of the maximal subgroupoid $\{ 0 , 1 \} \simeq [1]^\simeq \subset [1]$.  Then, in any solid commutative square
\[ \begin{tikzcd}
{[1]^\simeq} \arrow{r} \arrow[hook]{d} & \C \arrow{d} \\
{[1]} \arrow{r} \arrow[dashed]{ru} & \ho(\C)
\end{tikzcd} \]
there exists a connected space of dotted lifts by definition of the homotopy category.
\end{proof}

With \cref{rep maps in infty-cat localizn of rel 1-cat by zigzags} in hand, we now proceed to the fully general case of $\infty$-categorical localizations of relative \textit{$\infty$-categories}.

\begin{proof}[Proof of \cref{rep maps in localizn of rel infty-cat by zigzags}]
Observe that the morphism $(\R,\bW) \ra (\ho(\R),\ho(\bW))$ in $\RelCati$ induces a postcomposition
\[ [1] \xra{F} \loc{\R}{\bW} \ra \loc{\ho(\R)}{\ho(\bW)} \]
selecting a morphism in the $\infty$-categorical localization of the relative 1-category $(\ho(\R),\ho(\bW)) \in \RelCat$.  Hence, by \cref{rep maps in infty-cat localizn of rel 1-cat by zigzags}, we obtain a solid diagram
\[ \begin{tikzcd}[row sep=2cm, column sep=1.5cm]
& & & (\R,\bW) \arrow{ld} \arrow[maps to]{d} \\
& \word{m} \arrow{r} \arrow[maps to]{d} \arrow[dashed]{rru} & (\ho(\R),\ho(\bW)) & \loc{\R}{\bW} \arrow{ld} \arrow{d} \\
& \locL(\word{m}) \arrow{r} \arrow[dashed]{rru} & \loc{\ho(\R)}{\ho(\bW)} \arrow{d} & \ho(\loc{\R}{\bW}) \arrow{ld}[sloped]{\sim} \\
& & {\ho(\R)[\ho(\bW)^{-1}]} \\
{[1]} \arrow{rru} \arrow{ruu} \arrow{rruu}
\latearrow{crossing over, maps to}{2-3}{3-3}{}
\end{tikzcd} \]
for some relative word $\word{m} \in \Z$, in which
\begin{itemizesmall}
\item the lower right diagonal map is an equivalence by \cref{rnerves:loc and ho commute},
\item we moreover obtain the upper dotted arrow from \cref{rel word concatenation is a pushout} by induction, and
\item we define the lower dotted arrow to be its image under localization.
\end{itemizesmall}
Now, the resulting composite
\[ [1] \ra \locL(\word{m}) \ra \loc{\R}{\bW} \]
fits into a commutative diagram
\[ \begin{tikzcd}[row sep=1.5cm]
{[1]} \arrow{r} \arrow{d} & \locL(\word{m}) \arrow{r} \arrow{ld} & \loc{\R}{\bW} \arrow{d} \\
\loc{\ho(\R)}{\ho(\bW)} \arrow{r} & {\ho(\R)[\ho(\bW)^{-1}]} & \ho(\loc{\R}{\bW}) \arrow{l}{\sim}
\end{tikzcd} \]
in $\Cati$.  In particular, we have obtained a lift
\[ \begin{tikzcd}
& \loc{\R}{\bW} \arrow{d} \\
{[1]} \arrow[dashed]{ru} \arrow{r} & \ho(\loc{\R}{\bW})
\end{tikzcd} \]
of the composite
\[ [1] \xra{F} \loc{\R}{\bW} \ra \ho(\loc{\R}{\bW}) , \]
which must therefore be equivalent to $F$ itself by \cref{space of lifts of a map in htpy cat is connected}.  Thus, we obtain a diagram
\[ \begin{tikzcd}
& \word{m} \arrow[dashed]{r} \arrow[maps to]{d} & (\R,\bW) \arrow[maps to]{d} \\
& \locL(\word{m}) \arrow[dashed]{r} & \loc{\R}{\bW} \\
{[1]} \arrow{ru} \arrow{rru}
\end{tikzcd} \]
as in the proof of \cref{rep maps in infty-cat localizn of rel 1-cat by zigzags}, which proves the claim.
\end{proof}

Thus, zigzags play an important role not just in the theory of relative 1-categories and their 1-categorical localizations, but more generally in the theory of relative $\infty$-categories and their $\infty$-categorical localizations.

\subsection{Hammocks in relative $\infty$-categories}\label{subsection hammock sspaces}

For a general relative $\infty$-category $(\R,\bW)$, the representation of a morphism in $\loc{\R}{\bW}$ by a zigzag $\word{m} \ra (\R,\bW)$ guaranteed by \cref{rep maps in localizn of rel infty-cat by zigzags} is clearly far from unique.  Indeed, any morphism $\word{m}' \ra \word{m}$ in $\Z$ gives rise to a composite $\word{m}' \ra \word{m} \ra (\R,\bW)$ which presents the same morphism in $\loc{\R}{\bW}$: in other words, the morphisms in $\Z$ corepresent \textit{universal equivalence relations} between zigzags in relative $\infty$-categories (with respect to the morphisms that they represent upon localization).

In order to account for this over-representation, we are led to the following definition, the second of the two key concepts of this section, an analog of \cite[2.1]{DKCalc}.

\begin{defn}\label{define hammocks}
Suppose $(\R,\bW) \in \RelCati$, and suppose $x,y \in \R$.  We define the \bit{simplicial space of hammocks} (or alternatively the \bit{hammock simplicial space}) in $(\R,\bW)$ from $x$ to $y$ to be the colimit
\[ \homhamR ( x , y ) = \colim_{\word{m} \in \Z^{op}} \Nervei (\word{m}(x,y)) \in s\S . \]
\end{defn}

We will extend the hammock simplicial space construction further -- and in particular, justify its notation -- by constructing the \textit{hammock localization}
\[ \ham(\R,\bW) \in \CatsS \]
of $(\R,\bW)$ in \cref{section hammocks} (see \cref{justify hom notation}).

We now compare our hammock simplicial spaces of \cref{define hammocks} with Dwyer--Kan's classical hammock simplicial \textit{sets} (in relative 1-categories).

\begin{rem}\label{hammocks disagree}
Suppose that $(\R,\bW) \in \strrelcat$ is a relative category.  Then, \cite[Proposition 5.5]{DKCalc}, we have an identification
\[ \homhamdR(x,y) \cong \colim^{s\Set}_{\word{m} \in \Z^{op}} \Nerve ( \word{m}(x,y)) \]
of the classical simplicial \textit{set} of hammocks defined in \cite[2.1]{DKCalc} as an analogous colimit over the \textit{1-categorical} nerves of the categories of zigzags in $(\R,\bW)$ from $x$ to $y$.\footnote{It is not hard to see that the presence of the initial object $[\es]^\opobj \in \Z^{op}$ (which is what distinguishes this indexing category from $\II$) does not change this colimit.}  However, there are two reasons that this does not coincide with \cref{define hammocks}.
\begin{itemize}
\item The colimit computing $\homhamdR(x,y)$ is taken in the subcategory $s\Set \subset s\S$.  This inclusion (being a right adjoint) does not generally commute with colimits.
\item The functors $\strcat \xra{\Nerve} s\Set \hookra s\S$ and $\strcat \ra \Cati \xra{\Nervei} s\S$ do not generally agree, but are only related by a natural transformation
\[ \begin{tikzcd}
\strcat \arrow{r}{\Nerve}[swap, transform canvas={yshift=-1.2em}]{\LDarrow} \arrow{d} & s\Set \arrow[hook]{d}{\disc} \\
\Cati \arrow{r}[swap]{\Nervei} & s\S
\end{tikzcd} \]
in $\Fun(\strcat , s\S)$ (see \cref{rnerves:nerve vs nervei}).
\end{itemize}

On the other hand, these two constructions do at least participate in a diagram
\[ \begin{tikzcd}[column sep=2cm, row sep=2cm]
& s\Set \arrow[bend left]{rd}{\disc} \\
\Z^{op} \arrow[bend left]{ru}{\Nerve((-)(x,y))} \arrow[bend right=80, out=-80, in=-100]{rr}[transform canvas={yshift=4.5em}]{\Downarrow}[swap]{\Nervei((-)(x,y))} & & s\S
\end{tikzcd} \]
in $\Cati$, which induces a span
\[ \begin{tikzcd}[column sep=-1.5cm, row sep=1.5cm]
& \colim^{s\S}_{\word{m}^\opobj \in \Z^{op}} \disc(\Nerve(\word{m}(x,y))) \arrow{ld} \arrow{rd} \\
\homhamR(x,y) & & \disc \left( \homhamdR(x,y) \right)
\end{tikzcd} \]
in $s\S$.  We claim that this span lies in the subcategory $\bW_\KQ \subset s\S$, i.e.\! that it becomes an equivalence upon geometric realization; as we have a commutative triangle
\[ \begin{tikzcd}
s\Set \arrow[hook]{rr}{\disc} \arrow{rd}[sloped, swap, pos=0.65]{|{-}|} & & s\S \arrow{ld}[sloped, pos=0.65]{|{-}|} \\
& \S
\end{tikzcd} \]
in $\Cati$, this will imply that we have a canonical equivalence
\[ \left| \homhamR(x,y) \right| \simeq \left| \homhamdR(x,y) \right| \]
in $\S$.  We view this as a satisfactory state of affairs, since we are only ultimately interested in simplicial sets/spaces of hammocks as presentations of hom-spaces, anyways.

To see the claim, note first that since $|{-}| : s\S \ra \S$ is a left adjoint, it commutes with colimits, and so the left leg of the span lies in $\bW_\KQ$ by the fact that upon postcomposition with the geometric realization functor $|{-}| : s\S \ra \S$, the natural transformation
\[ \disc \circ \Nerve \ra \Nervei \]
in $\Fun(\strcat , s\S)$ becomes a natural equivalence
\[ |{-}| \circ \disc \circ \Nerve \ra |{-}| \circ \Nervei \]
in $\Fun(\strcat , \S)$ (again see \cref{rnerves:nerve vs nervei}).  By \cref{rnerves:groupoid-completion of CSSs}, these geometric realizations of colimits in $s\S$ both evaluate to
\[ \colim^\S_{\word{m}^\opobj \in \Z^{op}} \word{m}(x,y)^\gpd . \]

Now, in order to compute the geometric realization
\[ \left| \disc \left( \homhamdR(x,y) \right) \right| \simeq \left| \homhamdR(x,y) \right| , \]
we begin by observing that that the category $\Z$ has an evident Reedy structure, which one can verify has cofibrant constants, so that the dual Reedy structure on $\Z^{op}$ has fibrant constants.  Moreover, it is not hard to verify that the functor
\[ \Z^{op} \xra{\Nerve((-)(x,y))} s\Set \]
defines a cofibrant object of $\Fun(\Z^{op},s\Set_\KQ)_\Reedy$.  Hence, the colimit
\[ \homhamdR(x,y) \cong \colim^{s\Set}_{\word{m}^\opobj \in \Z^{op}} \Nerve(\word{m}(x,y)) \]
computes the homotopy colimit in $s\Set_\KQ$, i.e.\! the colimit of the composite
\[ \Z^{op} \xra{\Nerve((-)(x,y))} s\Set \xra{|{-}|} \loc{s\Set}{\bW_\KQ} \simeq \S . \]
The claim then follows from the string of equivalences
\[ |{-}| \circ \Nerve \simeq |{-}| \circ \disc \circ \Nerve \simeq |{-}| \circ \Nervei \simeq (-)^\gpd \]
in $\Fun(\strcat,\S)$ (again appealing to \cref{rnerves:groupoid-completion of CSSs}).
\end{rem}

\begin{rem}
Dwyer--Kan give a point-set definition of the hammock simplicial \textit{set} in \cite[2.1]{DKCalc}, and then prove it is isomorphic to the colimit indicated in \cref{hammocks disagree}.  However, working $\infty$-categorically, it is essentially impossible to make such an ad hoc definition.  Thus, we have simply \textit{defined} our hammock simplicial space as the colimit to which we would like it to be equivalent anyways.
\end{rem}

\subsection{Functoriality and gluing for zigzags}\label{subsection gluing zigzags}

In this subsection, we prove that $\infty$-categories of zigzags are suitably functorial for weak equivalences among source and target objects (see \cref{notation for half-doubly-pointed words}), and we use this to give a formula for an $\infty$-category of zigzags of type $[\word{m};\word{m}']$, the concatenation of two arbitrary relative words $\word{m} , \word{m}' \in \Z$ (see \cref{concatenated zigzags}).

Recall from \cref{rel word concatenation is a pushout} that concatenations of relative words compute pushouts in $\RelCati$.  This allows for inductive arguments, in which at each stage we freely adjoin a new morphism along either its source or its target.  For these, we will want to have a certain functoriality property for diagrams of this shape.  To describe it, let us first work in the special case of $\Cati$ (instead of $\RelCati$).  There, if for instance we have an $\infty$-category $\D'$ with a chosen object $d\in \D'$ and we use this to define a new $\infty$-category $\D$ as the pushout
\[ \begin{tikzcd}
\pt_\Cati \arrow{r}{t} \arrow{d}[swap]{d} & {[1]} \arrow{d} \\
\D' \arrow{r} & \D ,
\end{tikzcd} \]
then for any target $\infty$-category $\C$, the evaluation
\[ \Fun(\D,\C) \ra \Fun([1],\C) \xra{s} \C \]
will be a cartesian fibration by Corollary T.2.4.7.12 (applied to the functor $\Fun(\D',\C) \xra{\ev_d} \C$).  The following result is then an analog of this observation for \textit{relative} $\infty$-categories; note that there are now two types of ``freely adjoined morphisms'' we must consider.

\begin{lem}\label{ev at free target gives cocart fibn}
Let $(\I',\bW_{\I'}) \in \RelCati$, choose any $i \in \I'$, and suppose we are given any $(\R,\bW_\R) \in \RelCati$.

\begin{enumerate}

\item\label{add W}

\begin{enumeratesub}

\item\label{add W at s} If we form the pushout
\[ \begin{tikzcd}
\pt \arrow{r}{s} \arrow{d}[swap]{i} & {[\bW]} \arrow{d} \\
(\I',\bW_{\I'}) \arrow{r} & (\I,\bW_\I)
\end{tikzcd} \]
in $\RelCati$, then the composite restriction
\[ \Fun(\I,\R)^\bW \ra \Fun([\bW],\R)^\bW \xra{t} \bW_\R \]
is a cocartesian fibration.

\item\label{add W at t} Dually, if we form the pushout
\[ \begin{tikzcd}
\pt \arrow{r}{t} \arrow{d}[swap]{i} & {[\bW]} \arrow{d} \\
(\I',\bW_{\I'}) \arrow{r} & (\I,\bW_\I)
\end{tikzcd} \]
in $\RelCati$, then the composite restriction
\[ \Fun(\I,\R)^\bW \ra \Fun([\bW],\R)^\bW \xra{s} \bW_\R \]
is a cartesian fibration.
\end{enumeratesub}

\item\label{add A}


\begin{enumeratesub}

\item\label{add A at s} If we form the pushout
\[ \begin{tikzcd}
\pt \arrow{r}{s} \arrow{d}[swap]{i} & {[\any]} \arrow{d} \\
(\I',\bW_{\I'}) \arrow{r} & (\I,\bW_\I)
\end{tikzcd} \]
in $\RelCati$, then the composite restriction
\[ \Fun(\I,\R)^\bW \ra \Fun([\any],\R)^\bW \xra{t} \bW_\R \]
is a cocartesian fibration.

\item\label{add A at t} Dually, if we form the pushout
\[ \begin{tikzcd}
\pt \arrow{r}{t} \arrow{d}[swap]{i} & {[\any]} \arrow{d} \\
(\I',\bW_{\I'}) \arrow{r} & (\I,\bW_\I)
\end{tikzcd} \]
in $\RelCati$, then the composite restriction
\[ \Fun(\I,\R)^\bW \ra \Fun([\any],\R)^\bW \xra{s} \bW_\R \]
is a cartesian fibration.

\end{enumeratesub}
\end{enumerate}
\end{lem}

\begin{proof}
We first prove item \ref{add W}\ref{add W at t}.  Applying Corollary T.2.4.7.12 to the functor
\[ \Fun(\I',\R)^\bW \xra{i} \bW_\R \]
and noting that $\Fun([\bW],\R)^\bW \simeq \Fun([1],\bW_\R)$ (in a way compatible with the evaluation maps), we obtain that the composite restriction
\[ \Fun(\I,\R)^\bW \simeq \lim \left( \begin{tikzcd}
& \Fun([\bW],\R)^\bW \arrow{d}{t} \\
\Fun(\I',\R)^\bW \arrow{r}[swap]{i} & \bW_\R
\end{tikzcd} \right)
\ra \Fun([\bW],\R)^\bW \xra{s} \bW_\R \]
is a cartesian fibration, as desired.  The proof of item \ref{add W}\ref{add W at s} is completely dual.

We now prove item \ref{add A}\ref{add A at t}.  For this, consider the diagram
\[ \begin{tikzcd}
\Fun(\I,\R)^\bW \arrow{rr} \arrow{d} & & \Fun((\I')^\simeq , \bW_\R) \arrow{d} \\
\Fun(\I,\R)^{\bW @ s} \arrow{r} \arrow{dd}[swap]{s} & \Fun(\I,\R)^\Rel \arrow{r} \arrow{d} & \Fun(\I',\R)^\Rel \arrow{d}{i} \\
& \Fun([\any],\R)^\Rel \arrow{r}[swap]{t} \arrow{d}{s} & \R \\
\bW_\R \arrow{r} & \R
\end{tikzcd} \]
in which all small rectangles are pullbacks and in which we have introduced the ad hoc notation
\[ \Fun(\I,\R)^{\bW @ s} \subset \Fun(\I,\R)^\Rel \]
for the wide subcategory whose morphisms are those natural transformations whose component at $s \in [\any] \subset \I$ lies in $\bW_\R \subset \R$.  Observing that $\Fun([\any],\R)^\Rel \simeq \Fun([1],\R)$ (in a way compatible with the evaluation maps), it follows from applying Corollary T.2.4.7.12 to the functor
\[ \Fun(\I',\R)^\Rel \xra{i} \R \]
that the composite
\[ \Fun(\I,\R)^\Rel \ra \Fun([\any],\R)^\Rel \xra{s} \R \]
is a cartesian fibration, for which the cartesian morphisms are precisely those that are sent to equivalences under the restriction functor
\[ \Fun(\I,\R)^\Rel \ra \Fun(\I',\R)^\Rel . \]
Then, by Propositions T.2.4.2.3(2) and T.2.4.1.3(2), the functor
\[ \Fun(\I,\R)^{\bW @ s} \xra{s} \bW_\R \]
is also a cartesian fibration, for which any morphism that is sent to an equivalence under the composite
\[ \Fun(\I,\R)^{\bW @ s} \ra \Fun(\I,\R)^\Rel \ra \Fun(\I',\R)^\Rel \]
is cartesian.  Now, for any map $x' \xra{\varphi} x$ in $\bW_\R$ and any object
\[ G \in \left( \pt_\Cati \underset{x,\bW_\R,s}{\times} \Fun(\I,\R)^{\bW @ s} \right) , \]
there clearly exists such a cartesian morphism
\[ ( F \xra{\tilde{\varphi}} G ) \in \left( \Fun \left( [1],\Fun(\I,\R)^{\bW @ s} \right) \underset{\Fun([1],s) , \Fun([1],\bW_\R), \varphi}{\times} \pt_\Cati \right) \]
(which can easily be constructed using the definition of $(\I,\bW_\I)$ as a pushout).  Moreover, since by definition $\R^\simeq \subset \bW_\R$, it follows that this is in fact a morphism in the (wide) subcategory $\Fun(\I,\R)^\bW \subset \Fun(\I,\R)^{\bW @ s}$.  Hence, we obtain a diagram
\[ \begin{tikzcd}
\left( \Fun(\I,\R)^\bW \right)_{/\tilde{\varphi}} \arrow{r} \arrow{d} & \left( \Fun(\I,\R)^{\bW @ s} \right)_{/\tilde{\varphi}} \arrow{r} \arrow{d} & (\bW_\R)_{/\varphi} \arrow{d} \\
\left( \Fun(\I,\R)^\bW \right)_{/G} \arrow{r} & \left( \Fun(\I,\R)^{\bW @ s} \right)_{/G} \arrow{r} & (\bW_\R)_{/x}
\end{tikzcd} \]
in $\Cati$, in which the right square is a pullback since $\tilde{\varphi}$ is a cartesian morphism.  Moreover, again using the fact that $\R^\simeq \subset \bW_\R$, it is easy to check that the left square is also a pullback.  So the entire rectangle is a pullback, and hence $\tilde{\varphi}$ is also a cartesian morphism for the functor
\[ \Fun(\I,\R)^\bW \xra{s} \bW_\R . \]
From here, it follows from the fact that $\Fun(\I,\R)^\bW \subset \Fun(\I,\R)^{\bW @ s}$ is a subcategory that this functor is indeed a cartesian fibration.  The proof of item \ref{add A}\ref{add A at s} is completely dual.
\end{proof}

Given an arbitrary doubly-pointed relative $\infty$-catetgory $(\I,\bW_\I) \in \RelCatip$ some relative $\infty$-category $(\R,\bW_\R) \in \RelCati$ which we consider to be doubly-pointed via some choice $x,y \in \R$ of a pair of objects, we will be interested in the functoriality of the construction
\[ \Funp((\I,\bW_\I),((\R,\bW_\R),x,y))^\bW \in \Cati \]
in the variable $x \in \bW$ but for a fixed choice of $y \in \bW$ (or vice versa).  This functoriality will be expressed by a variant of \cref{ev at free target gives cocart fibn}.  However, in order to accommodate the fixing of just one of the two chosen objects, we must first introduce the following notation.

\begin{notn}
Let $\I \in \RelCatip$, let $(\R,\bW) \in \RelCati$, and let $x, y \in \R$.  Then, we write
\[ \left( \Funps(\I,\R)^\Rel , \Funps(\I,\R)^\bW \right) = \lim \left( \begin{tikzcd}
& \left( \Fun(\I,\R)^\Rel , \Fun(\I,\R)^\bW \right) \arrow{d}{s} \\
\pt_\RelCati \arrow{r}[swap]{x} & (\R,\bW)
\end{tikzcd} \right) \]
and
\[ \left( \Funpt(\I,\R)^\Rel , \Funpt(\I,\R)^\bW \right) = \lim \left( \begin{tikzcd}
& \left( \Fun(\I,\R)^\Rel , \Fun(\I,\R)^\bW \right) \arrow{d}{t} \\
\pt_\RelCati \arrow{r}[swap]{y} & (\R,\bW)
\end{tikzcd} \right) . \]
\end{notn}

We now give a ``half-doubly-pointed'' variant of \cref{ev at free target gives cocart fibn}, but stated only in the special case that we will need.

\begin{lem}\label{ev half-doubly-pointed}
Let $\word{m} \in \Z$, let $(\R,\bW) \in \RelCati$, and let $x,y \in \R$.
\begin{enumerate}

\item\label{ev at s} The functor $\Funpt(\word{m},\R)^\bW \xra{s} \bW$
\begin{enumeratesub}
\item\label{ev at s with W} is a cocartesian fibration if $\word{m}$ begins with $\bW^{-1}$, and
\item\label{ev at s with A} is a cartesian fibration if $\word{m}$ begins with $\any$.
\end{enumeratesub}

\item\label{ev at t} The functor $\Funps(\word{m},\R)^\bW \xra{t} \bW$
\begin{enumeratesub}
\item\label{ev at t with W} is a cartesian fibration if $\word{m}$ ends with $\bW^{-1}$, and
\item\label{ev at t with A} is a cocartesian fibration if $\word{m}$ ends with $\any$.
\end{enumeratesub}

\end{enumerate}
\end{lem}

\begin{proof}
If we simply have $\word{m} = [\any]$ or $\word{m} = [\bW^{-1}]$ then these statements follow trivially from \cref{ev at free target gives cocart fibn}, so let us assume that the relative word $\word{m}$ has length greater than 1.

To prove item \ref{ev at t}\ref{ev at t with W}, suppose that $\word{m} = [\word{m}' ; \bW^{-1}]$.  Then we have a pullback square
\[ \begin{tikzcd}
\Funps(\word{m},\R)^\bW \arrow{r} \arrow{d} & \Fun([\bW^{-1}],\R)^\bW \arrow{d}{s_{[\bW^{-1}]}} \\
\Funps(\word{m}',\R)^\bW \arrow{r}[swap]{t_{\word{m}'}} & \bW
\end{tikzcd} \]
which, making the identification of $[\bW^{-1}]$ with $[\bW]$ in a way which switches the source and target objects, is equivalently a pullback square
\[ \begin{tikzcd}
\Funps(\word{m},\R)^\bW \arrow{r} \arrow{d} & \Fun([\bW],\R)^\bW \arrow{d}{t_{[\bW]}} \\
\Funps(\word{m}',\R)^\bW \arrow{r}[swap]{t_{\word{m}'}} & \bW .
\end{tikzcd} \]
From here, the proof parallels that of \cref{ev at free target gives cocart fibn}\ref{add W}\ref{add W at t}, only now we apply Corollary T.2.4.7.12 to the functor
\[ \Funps(\word{m}',\R)^\bW \xra{t_{\word{m}'}} \bW . \]
The proof of item \ref{ev at s}\ref{ev at s with W} is completely dual.

To prove item \ref{ev at s}\ref{ev at s with A}, let us now suppose that $\word{m} = [\any;\word{m}']$.  Then we have a diagram
\[ \begin{tikzcd}
\Funpt(\word{m},\R)^\bW \arrow{rr} \arrow{d} & & \Funpt((\word{m}')^\simeq , \bW) \arrow{d} \\
\Funpt(\word{m},\R)^{\bW @ s} \arrow{r} \arrow{dd}[swap]{s} & \Funpt(\word{m},\R)^\Rel \arrow{r} \arrow{d} & \Funpt(\word{m}',\R)^\Rel \arrow{d}{s_{\word{m}'}} \\
& \Fun([\any],\R)^\bW \arrow{r}[swap]{t_{[\any]}} \arrow{d}{s_{[\any]}} & \R \\
\bW \arrow{r} & \R
\end{tikzcd} \]
in which all small rectangles are pullbacks, almost identical to that of the proof of \cref{ev at free target gives cocart fibn}\ref{add A}\ref{add A at t}.  From here, the proof proceeds in a completely analogous way to that one.  The proof of item \ref{ev at t}\ref{ev at t with A} is completely dual.
\end{proof}

\cref{ev half-doubly-pointed}, in turn, enables us to make the following definitions.

\begin{notn}\label{notation for half-doubly-pointed words}
Let $\word{m} \in \Z$, let $(\R,\bW) \in \RelCati$, and let $x,y \in \R$.
\begin{itemize}

\item If $\word{m}$ begins with $\bW^{-1}$, we write
\[ \bW \xra{\word{m}(-,y)} \Cati \]
for the functor classifying the cocartesian fibration of \cref{ev half-doubly-pointed}\ref{ev at s}\ref{ev at s with W}.  On the other hand, if $\word{m}$ begins with $\any$, we write
\[ \bW^{op} \xra{\word{m}(-,y)} \Cati \]
for the functor classifying the cartesian fibration of \cref{ev half-doubly-pointed}\ref{ev at s}\ref{ev at s with A}.

\item If $\word{m}$ ends with $\bW^{-1}$, we write
\[ \bW^{op} \xra{\word{m}(x,-)} \Cati \]
for the functor classifying the cartesian fibration of \cref{ev half-doubly-pointed}\ref{ev at t}\ref{ev at t with W}.  On the other hand, if $\word{m}$ ends with $\any$, we write
\[ \bW \xra{\word{m}(x,-)} \Cati \]
for the functor classifying the cocartesian fibration of \cref{ev half-doubly-pointed}\ref{ev at t}\ref{ev at t with A}.

\item By convention and for convenience, if $\word{m} = [\es] \in \Z$ is the empty relative word (which defines the terminal relative $\infty$-category), we let both $\word{m}(x,-)$ and $\word{m}(-,y)$ denote either functor
\[ \bW \xra{\const(\pt_\Cati)} \Cati \]
or
\[ \bW^{op} \xra{\const(\pt_\Cati)} \Cati . \]

\end{itemize}
\end{notn}

Using \cref{notation for half-doubly-pointed words}, we now express the $\infty$-category $[\word{m};\word{m}']_{(\R,\bW)}(x,y)$ of zigzags in $(\R,\bW)$ from $x$ to $y$ of the concatenated zigzag type $[\word{m};\word{m}']$ in terms of the two-sided Grothendieck construction (see \cref{gr:define two-sided Gr}).  This is an analog of \cite[9.4]{DKCalc}.\footnote{In the statement of \cite[9.4]{DKCalc}, the third appearance of $\word{m}$ should actually be $\word{m}'$.}

\begin{lem}\label{concatenated zigzags}
Let $\word{m},\word{m}' \in \Z$.  Then for any $(\R,\bW) \in \RelCati$ and any $x,y \in \R$, we have an equivalence
\[ [\word{m};\word{m}'](x,y) \simeq \left\{ \begin{array}{ll}
\Gr \left( \word{m}'(-,y) , \bW , \word{m}(x,-) \right) , & \textup{$\word{m}$ ends with $\any$ and $\word{m}'$ begins with $\any$} \\
\Gr \left( \word{m}(x,-) , \bW , \word{m}'(-,y) \right) , & \textup{$\word{m}$ ends with $\bW^{-1}$ and $\word{m}'$ begins with $\bW^{-1}$} \\
\Gr \left( \const(\pt) , \bW , \left( \word{m}(x,-) \times \word{m}'(-,y) \right) \right) , & \textup{$\word{m}$ ends with $\any$ and $\word{m}'$ begins with $\bW^{-1}$} \\
\Gr \left( \left( \word{m}(x,-) \times \word{m}'(-,y) \right) , \bW , \const(\pt) \right) , & \textup{$\word{m}$ ends with $\bW^{-1}$ and $\word{m}'$ begins with $\any$.}
\end{array} \right. \]
which is natural in $((\R,\bW),x,y) \in \RelCatip$.
\end{lem}

\begin{proof}
Recall from \cref{rel word concatenation is a pushout} that we have a pushout square
\[ \begin{tikzcd}
\pt_\RelCati \arrow{r}{s} \arrow{d}[swap]{t} & \word{m}' \arrow{d} \\
\word{m} \arrow{r} & {[\word{m};\word{m}']}
\end{tikzcd} \]
in $\RelCati$, through which $[\word{m};\word{m}']$ acquires its source object from $\word{m}$ and its target object from $\word{m}'$.  This gives rise to a string of equivalences
\begin{align*}
[\word{m};\word{m}'](x,y)
= \Funp ( [\word{m};\word{m}'] , \R)^\bW
& \simeq \lim \left( \begin{tikzcd}[ampersand replacement=\&]
\& \& \& \pt_\Cati \arrow{d}{y} \\
\& \& \Fun(\word{m}',\R)^\bW \arrow{r}[swap]{t} \arrow{d}{s} \& \bW \\
\& \Fun(\word{m},\R)^\bW \arrow{r}[swap]{t} \arrow{d}{s} \& \bW \\
\pt_\Cati \arrow{r}[swap]{x} \& \bW
\end{tikzcd} \right)
\\
& \simeq \lim \left( \begin{tikzcd}[ampersand replacement=\&]
\& \Funpt(\word{m}',\R)^\bW \arrow{d}{s} \\
\Funps(\word{m},\R)^\bW \arrow{r}[swap]{t} \& \bW
\end{tikzcd} \right)
\end{align*}
in $\Cati$.  From here, the first and second cases follow from \cref{ev half-doubly-pointed}, \cref{notation for half-doubly-pointed words}, and \cref{gr:define two-sided Gr}, while the third and fourth cases follow by additionally appealing to \cref{gr:Gr preserves terminal objects} and \cref{gr:fiber prods of cocart fibns is Gr of product}.
\end{proof}

\section{Homotopical three-arrow calculi in relative $\infty$-categories}\label{section three-arrow calculi}

In the previous section, given a relative $\infty$-category $(\R,\bW)$, we introduced the \textit{hammock simplicial space}
\[ \homhamR(x,y) \in s\S \]
for two given objects $x,y \in \R$.  The definition of this simplicial space is fairly explicit, but it is nevertheless quite large.  In this section, we show that under a certain condition -- namely, that $(\R,\bW)$ admits a \textit{homotopical three-arrow calculus} -- we can at least recover this simplicial space up to weak equivalence in $s\S_\KQ$ (i.e.\! we can recover its geometric realization) froma much smaller simplicial space, in fact from one of the constituent simplicial spaces in its defining colimit.  This condition is often satisfied in practice; for example, it holds when $(\R,\bW)$ admits the additional structure of a \textit{model $\infty$-category} (see \cref{fundthm:model infty-cats have calculi}).

This section is organized as follows.
\begin{itemize}

\item In \cref{subsection statement of fund thm of htpical three-arrow calculi}, we define what it means for a relative $\infty$-category to admit a homotopical three-arrow calculus, and we state the \textit{fundamental theorem of homotopical three-arrow calculi} (\ref{calculus gives reduction}) described above.

\item In \cref{subsection rel infty-cats for three-arrow calculi}, in preparation for the proof of \cref{calculus gives reduction}, we assemble some auxiliary results regarding relative $\infty$-categories.

\item In \cref{subsection co-ends for htpical three-arrow calculi}, in preparation for the proof of \cref{calculus gives reduction}, we assemble some auxiliary results regarding ends and coends.

\item In \cref{subsection proof of htpical three-arrow calculi}, we give the proof of \cref{calculus gives reduction}.

\end{itemize}

\subsection{The fundamental theorem of homotopical three-arrow calculi}\label{subsection statement of fund thm of htpical three-arrow calculi}

We begin with the main definition of this section, whose terminology will be justified by \cref{calculus gives reduction}; it is a straightforward generalization of \cite[Definition 4.1]{LowMG}, which is itself a minor variant of \cite[6.1(i)]{DKCalc}.

\begin{defn}\label{define calculus}
Let $(\R,\bW) \in \RelCati$.  We say that $(\R,\bW)$ admits a \bit{homotopical three-arrow calculus} if for all $x,y \in \R$ and for all $i,j \geq 1$, the map
\[ [ \bW^{-1} ; \any^{\circ i} ; \bW^{-1} ; \any^{\circ j} ; \bW^{-1} ] \ra [ \bW^{-1} ; \any^{\circ i} ; \any^{\circ j} ; \bW^{-1} ] \]
in $\Z \subset \RelCatp$ obtained by collapsing the middle weak equivalence induces a map
\[ \Funp ( [\bW^{-1} ; \any^{\circ i} ; \any^{\circ j} ; \bW^{-1} ] , \R)^\bW \ra \Funp ( [ \bW^{-1} ; \any^{\circ i} ; \bW^{-1} ; \any^{\circ j} ; \bW^{-1} ] , \R)^\bW \]
in $\bW^\Cati_\Thomason \subset \Cati$ (i.e.\! it becomes an equivalence upon applying the groupoid completion functor $(-)^\gpd : \Cati \ra \S$).
\end{defn}

\begin{notn}\label{notn 3}
Since it will appear repeatedly, we make the abbreviation $\word{3} = [ \bW^{-1} ; \any ; \bW^{-1} ]$ for the relative word
\[ \begin{tikzcd}
s & \bullet \arrow{l}[swap]{\approx} \arrow{r} & \bullet & t \arrow{l}[swap]{\approx} .
\end{tikzcd} \]
\end{notn}

\begin{defn}\label{defn 3-arrow zigzags}
For any relative $\infty$-category $(\R,\bW)$ and any objects $x,y \in \R$, we will refer to
\[ \word{3}(x,y) = \Funp(\word{3},\R)^\bW \in \Cati \]
as the $\infty$-category of \bit{three-arrow zigzags} in $\R$ from $x$ to $y$.
\end{defn}

We now state the \bit{fundamental theorem of homotopical three-arrow calculi}, an analog of \cite[Proposition 6.2(i)]{DKCalc}; we will give its proof in \cref{subsection proof of htpical three-arrow calculi}.

\begin{thm}\label{calculus gives reduction}
If $(\R,\bW) \in \RelCati$ admits a homotopical three-arrow calculus, then for any $x,y \in \R$, the natural map
\[ \Nervei(\word{3}(x,y)) \ra \homhamR(x,y) \]
in $s\S$ becomes an equivalence under the geometric realization functor $|{-}| : s\S \ra \S$.
\end{thm}

\subsection{Supporting material: relative $\infty$-categories}\label{subsection rel infty-cats for three-arrow calculi}

In this subsection, we give two results regarding relative $\infty$-categories which will be used in the proof of \cref{calculus gives reduction}.  Both concern \textit{corepresentation}, namely the effect of the functor
\[ \RelCatmp \xra{\Fun(-,\R)^\bW} \Cati \]
on certain data in $\RelCatmp$ (for a given relative $\infty$-category $(\R,\bW)$).

\begin{lem}\label{nat w.e. induces nat trans}
Given a pair of maps $\I \rra \J$ in $\RelCatimp$, a morphism between them in $\Funmp(\I,\J)^\bW$ induces, for any $(\R,\bW) \in \RelCatimp$, a natural transformation between the two induced functors
\[ \Funmp(\J,\R)^\bW \rra \Funmp(\I,\R)^\bW . \]
\end{lem}

\begin{proof}
First of all, the morphism in $\Funmp(\I,\J)^\bW$ is selected by a map $[1] \ra \Funmp(\I,\J)^\bW$; this is equivalent to a map
\[ [1]_\bW \ra \left( \Funmp(\I,\J)^\Rel , \Funmp(\I,\J)^\bW \right) \]
in $\RelCati$, which is adjoint to a map
\[ \I \tensoring [1]_\bW \ra \J \]
in $\RelCatimp$.  Then, for any $(\R,\bW) \in \RelCatimp$, composing with this map yields a functor
\begin{align*}
\Funmp(\J,\R)^\bW
& \ra \Funmp(\I \tensoring [1]_\bW , \R)^\bW \\
& \simeq \Fun \left( [1]_\bW , \left( \Funmp(\I,\R)^\Rel , \Funmp(\I,\R)^\bW \right) \right) \\
& \simeq \Fun \left( [1] , \Funmp(\I,\R)^\bW \right) ,
\end{align*}
which is adjoint to a map
\[ [1] \times \Funmp(\J,\R)^\bW \ra \Funmp(\I,\R)^\bW , \]
which selects a natural transformation between the two induced functors
\[ \Funmp(\J,\R)^\bW \rra \Funmp(\I,\R)^\bW , \]
as desired.
\end{proof}

\begin{lem}\label{compose w.e.'s}
Let $(\I,\bW_\I) \in \RelCatimp$, and form any pushout diagram
\[ \begin{tikzcd}
{[\bW]} \arrow{r} \arrow{d} & (\I,\bW_\I) \arrow{d} \\
{[\bW^{\circ 2}]} \arrow{r} & (\J,\bW_\J)
\end{tikzcd} \]
in $\RelCatmp$, where the left map is the unique map in $\RelCatp$.  Note that the two possible retractions $[\bW^{\circ 2}] \rra [\bW]$ in $\RelCatp$ of the given map induce retractions $(\J,\bW_\J) \rra (\I,\bW_\I)$ in $\RelCatimp$.  Then, for any $(\R,\bW_\R) \in \RelCatmp$, the induced map
\[ \Funmp(\J,\R)^\bW \ra \Funmp(\I,\R)^\bW \]
which becomes an equivalence under the functor $(-)^\gpd : \Cati \ra \S$, with inverse given by either map
\[ \left( \Funmp(\I,\R)^\bW \right)^\gpd \rra \left( \Funmp(\J,\R)^\bW \right)^\gpd \]
in $\S$ induced by one of the given retractions.
\end{lem}

\begin{proof}
Note that both composites
\[ [\bW^{\circ 2}] \rra [\bW] \ra [\bW^{\circ 2}] \]
(of one of the two possible retractions followed by the given map) are connected to $\id_{[\bW^{\circ 2}]}$ by a map in
\[ \Funp([\bW^{\circ 2}],[\bW^{\circ 2}])^\bW . \]
In turn, both composites
\[ (\J,\bW_\J) \rra (\I,\bW_\I) \ra (\J,\bW_\J) \]
are connected to $\id_{(\J,\bW_\J)}$ by a map in $\Funmp(\J,\J)^\bW$.  Hence, the result follows from Lemmas \ref{nat w.e. induces nat trans} \and \Cref{rnerves:nat trans induces equivce betw maps on gpd-complns}.
\end{proof}

\subsection{Supporting material: co/ends}\label{subsection co-ends for htpical three-arrow calculi}

In this subsection, we give a few results regarding ends and coends which will be used in the proof of \cref{calculus gives reduction}.  For a brief review of these universal constructions in the $\infty$-categorical setting, we refer the reader to \cite[\sec 2]{GHN}.

We begin by recalling a formula for the space of natural transformations between two functors.

\begin{lem}\label{nat trans as end}
Given any $\C,\D \in \Cati$ and any $F,G \in \Fun(\C,\D)$, we have a canonical equivalence
\[ \hom_{\Fun(\C,\D)}(F,G) \simeq \int_{c \in \C} \hom_\D(F(c),G(c)) . \]
\end{lem}

\begin{proof}
This appears as \cite[Proposition 2.3]{SaulHodge} (and as \cite[Proposition 5.1]{GHN}).
\end{proof}

We now prove a ``ninja Yoneda lemma''.\footnote{The name is apparently due to Leinster (see \cite[Remark 2.2]{FoscoCoend}).}

\begin{lem}\label{ninja yoneda}
If $\C \in \Cati$ is an $\infty$-category equipped with a tensoring $- \tensoring - : \C \times \S \ra \C$, then for any functor $\I^{op} \xra{F} \C$, we have an equivalence
\[ F(-) \simeq \int^{i \in \I} F(i) \tensoring \hom_\I(- , i) \]
in $\Fun(\I^{op},\C)$.
\end{lem}

\begin{proof}
For any test objects $j \in \I^{op}$ and $Y \in \C$, we have a string of natural equivalences
\begin{align*}
\hom_\C \left( \int^{i \in \I} F(i) \tensoring \hom_\I(j,i) , Y \right)
& \simeq \int_{i \in \I} \hom_\C ( F(i) \tensoring \hom_\I(j,i) , Y ) \\
& \simeq \int_{i \in \I} \hom_\S ( \hom_\I(j,i) , \hom_\C(F(i) , Y)) \\
& \simeq \hom_{\Fun(\I,\S)}( \hom_\I(j,-) , \hom_\C(F(-),Y)) \\
& \simeq \hom_\C(F(j),Y) ,
\end{align*}
where the first line follows from the definition of a coend as a colimit (see e.g.\! \cite[Definition 2.5]{GHN}), the second line uses the tensoring, the third line follows from \cref{nat trans as end}, and the last line follows from the usual Yoneda lemma (Proposition T.5.1.3.1).  Hence, again by the Yoneda lemma, we obtain an equivalence
\[ F(j) \simeq \int^{i \in \I} F(i) \tensoring \hom_\J(j,i) \]
which is natural in $j \in \I^{op}$.
\end{proof}

Then, we have the following result on the preservation of colimits.\footnote{\cref{wtd colimit is left adjoint in the weight} is actually implicitly about weighted colimits (see \cite[Definition 2.7]{GHN}).}

\begin{lem}\label{wtd colimit is left adjoint in the weight}
If $\C \in \Cati$ is an $\infty$-category equipped with a tensoring $- \tensoring - : \C \times \S \ra \C$, then for any functor $\I^{op} \xra{F} \C$, the functor
\[ \Fun(\I,\S) \xra{ \int^{i \in \I} F(i) \tensoring (-)(i)} \C \]
is a left adjoint.
\end{lem}

\begin{proof}
It suffices to check that for every $c \in \C$, the functor
\[ \Fun(\I,\S)^{op} \xra{\hom_\C \left( \int^{i \in \I} F(i) \tensoring (-)(i) , c \right)} \S \]
is representable.  For this, given any $W \in \Fun(\I,\S)$ we compute that
\begin{align*}
\hom_\C \left( \int^{i \in \I} F(i) \tensoring W(i) , c \right)
& \simeq \int_{i \in \I} \hom_\C(F(i) \tensoring W(i) , c) \\
& \simeq \int_{i \in \I} \hom_\S(W(i) , \hom_\C(F(i),c)) \\
& \simeq \hom_{\Fun(\I,\S)} ( W , \hom_\C(F(-),c)),
\end{align*}
where the first line follows from the definition of a co/end as a co/limit (again see e.g.\! \cite[Definition 2.5]{GHN}), the second line uses the tensoring, and the last line follows from \cref{nat trans as end}.
\end{proof}

\subsection{The proof of \cref{calculus gives reduction}}\label{subsection proof of htpical three-arrow calculi}

Having laid out the necessary supporting material in the previous two subsection, we now proceed to prove the fundamental theorem of homotopical three-arrow calculi (\ref{calculus gives reduction}).  This proof is based closely on that of \cite[Proposition 6.2(i)]{DKCalc}, although we give many more details (recall \cref{intro rem more details than DK}).

\begin{proof}[Proof of \cref{calculus gives reduction}]
We will construct a commutative diagram
\[ \begin{tikzcd}
\left| \Nervei(\word{3}(x,y)) \right| \arrow{r}{|\beta|}[swap]{\sim} \arrow{d}[swap]{|\alpha|} & \left| \colim_{\word{m} \in \Z^{op}} \Nervei(G(\word{m})(x,y)) \right| \arrow{d}{|\ul{\psi}|}[sloped, anchor=north]{\sim} \\
\left| \colim_{\word{m} \in \Z^{op}} \Nervei(\word{m}(x,y)) \right| \arrow{r}{|\ul{\varphi}|} & \left| \colim_{\word{m} \in \Z^{op}} \Nervei(F(\word{m})(x,y)) \right| \arrow[bend left=20]{l}{|\rho|}
\end{tikzcd} \]
in $\S$, i.e.\! a commutative square in which the bottom arrow is equipped with a retraction and in which moreover the top and right map are equivalences.  Note that by definition, the object on the bottom left is precisely $\left| \homhamR(x,y) \right|$; the left map will be the natural map referred to in the statement of the result.  The equivalences in $\S$ satisfy the two-out-of-six property, and applying this to the composable sequence of arrows $[|\alpha|;|\ul{\varphi}|;|\rho|]$, we deduce that $|\alpha|$ is also an equivalence, proving the claim.

We will accomplish this by running through the following sequence of tasks.
\begin{enumerate}
\item Define the two objects on the right.
\item Define the maps in the diagram.
\item Explain why the square commutes.
\item Explain why $|\rho|$ gives a retraction of $|\ul{\varphi}|$.
\item Explain why the map $|\beta|$ is an equivalence.
\item Explain why the map $|\ul{\psi}|$ is an equivalence.
\end{enumerate}

We now proceed to accomplish these tasks in order.
\begin{enumerate}

\item We define endofunctors $F, G \in \Fun(\Z , \Z)$ by the formulas
\[ F(\word{m}) = [-1;\word{m};-1] \]
and
\[ G(\word{m}) = [-1 ; \any^{\circ |\word{m}|_\any} ; -1] . \]
Then, the object in the upper right is given by
\[ \left| \colim \left( \Z^{op} \xra{G^{op}} \Z^{op} \xra{\Nervei((-)(x,y))} s\S \right) \right| , \]
and the object in the bottom right is given by
\[ \left| \colim \left( \Z^{op} \xra{F^{op}} \Z^{op} \xra{\Nervei((-)(x,y))} s\S \right) \right| . \]

\item We define the two evident natural transformations $F \xra{\varphi} \id_{\Z}$ (given by collapsing the two newly added copies of $[\bW^{-1}]$) and $F \xra{\psi} G$ (given by collapsing all internal copies of $[\bW^{-1}]$) in $\Fun(\Z,\Z)$; these induce natural transformations $\id_{\Z^{op}} \xra{\varphi^{op}} F^{op}$ and $G^{op} \xra{\psi^{op}} F^{op}$ in $\Fun(\Z^{op},\Z^{op})$.\footnote{Recall that the involution $(-)^{op} : \Cati \ra \Cati$ is \textit{contravariant} on 2-morphisms.}  We then define the maps in the diagram as follows.

\begin{itemize}

\item The left map is obtained by taking the geometric realization of the inclusion
\[ \Nervei(\word{3}(x,y)) \xra{\alpha} \homhamR(x,y) = \colim_{\word{m} \in \Z^{op}} \Nervei(\word{m}(x,y)) \]
into the colimit at the object $\word{3} \in \Z^{op}$.

\item The top map is obtained by taking the geometric realization of the inclusion
\[ \Nervei(\word{3}(x,y)) \simeq \Nervei(G([\any])(x,y)) \xra{\beta} \colim_{\word{m} \in \Z^{op}} \Nervei(G(\word{m})(x,y)) \]
into the colimit at the object $[\any] \in \Z^{op}$.  (Note that $\word{3} \cong G([\any])$ in $\Z^{op}$.)

\item The right map is obtained by taking the geometric realization of the map
\[ \colim_{\word{m} \in \Z^{op}} \Nervei(G(\word{m})(x,y)) \xra{\ul{\psi}} \colim_{\word{m} \in \Z^{op}} \Nervei(F(\word{m})(x,y)) \]
on colimits induced by the natural transformation $\id_{\Nervei((-)(x,y))} \circ \psi^{op}$ in $\Fun(\Z^{op},s\S)$.

\item The bottom map in the square (i.e.\! the straight bottom map) is obtained by taking the geometric realization of the map
\[ \homhamR(x,y) = \colim_{\word{m} \in \Z^{op}} \Nervei(\word{m}(x,y)) \xra{\ul{\varphi}} \colim_{\word{m} \in \Z^{op}} \Nervei(F(\word{m})(x,y)) \]
on colimits induced by the natural transformation $\id_{\Nervei((-)(x,y))} \circ \varphi^{op}$ in $\Fun(\Z^{op},s\S)$.

\item The curved map is obtained by taking the geometric realization of the map
\[ \colim_{\word{m} \in \Z^{op}} \Nervei(F(\word{m})(x,y)) \xra{\rho} \colim_{\word{m} \in \Z^{op}} \Nervei(\word{m}(x,y)) = \homhamR(x,y) \]
on colimits induced by the functor
\[ \Fun(\Z^{op},s\S) \xla{- \circ F^{op}} \Fun(\Z^{op},s\S) . \]

\end{itemize}

\item The upper composite in the square is given by the geometric realization of the composite
\[ \Nerve(\word{3}(x,y)) \simeq \Nervei(G([\any])(x,y)) \xra[\sim]{\Nervei((\psi^{op}_{[\any]})(x,y))} \Nervei(F([\any])(x,y)) \ra \colim_{\word{m} \in \Z^{op}} \Nervei(F(\word{m})(x,y)) \]
of the equivalence induced by the component of $\psi^{op}$ at the object $[\any] \in \Z^{op}$ (which is an isomorphism in $\Z^{op}$) followed by the inclusion into the colimit at $[\any]$.  So, via the (unique) identification $\word{3} \cong F([\any])$, we can identify this composite with the inclusion into the colimit at $[\any] \in \Z^{op}$.

Meanwhile, the lower composite in the square is given by the geometric realization of the composite
\[ \Nervei(\word{3}(x,y)) \xra{\Nervei((\varphi^{op}_{\word{3}})(x,y))} \Nervei(F(\word{3})(x,y)) \ra \colim_{\word{m} \in \Z^{op}} \Nervei(F(\word{m})(x,y)) \]
of the map induced by the component of $\varphi^{op}$ at $\word{3}$ followed by the inclusion into the colimit at $\word{3}$.

Now, the map $F(\word{3}) \xra{\varphi_\word{3}} \word{3}$ in $\Z$ is given by
\[ \begin{tikzcd}
s_{F(\word{3})} \arrow[maps to]{rd} & \bullet \arrow{l}[swap]{\approx} \arrow[maps to]{d} & \bullet \arrow{l}[swap]{\approx} \arrow{r} \arrow[maps to]{d} & \bullet \arrow[maps to]{d} & \bullet \arrow{l}[swap]{\approx} \arrow[maps to]{d} & t_{F(\word{3})} \arrow{l}[swap]{\approx} \arrow[maps to]{ld} \\
& s_\word{3} & \bullet \arrow{r} \arrow{l}[swap]{\approx} & \bullet & t_\word{3} . \arrow{l}[swap]{\approx}
\end{tikzcd} \]
On the other hand, applying $F$ to the unique map $\word{3} \xra{\gamma} [\any]$ in $\Z$, we obtain a map $F(\word{3}) \xra{F(\gamma)} F([\any]) \cong \word{3}$ in $\Z$ given by
\[ \begin{tikzcd}
s_{F(\word{3})} \arrow[maps to]{rd} & \bullet \arrow{l}[swap]{\approx} \arrow[maps to]{rd} & \bullet \arrow{l}[swap]{\approx} \arrow{r} \arrow[maps to]{d} & \bullet \arrow[maps to]{d} & \bullet \arrow{l}[swap]{\approx} \arrow[maps to]{ld} & t_{F(\word{3})} \arrow{l}[swap]{\approx} \arrow[maps to]{ld} \\
& s_\word{3} & \bullet \arrow{r} \arrow{l}[swap]{\approx} & \bullet & t_\word{3} . \arrow{l}[swap]{\approx}
\end{tikzcd} \]
which corepresents a map
\[ \Nervei(\word{3}(x,y)) \simeq \Nervei(F([\any])(x,y)) \xra{\Nervei((F(\gamma))(x,y))} \Nervei(F(\word{3})(x,y)) \]
in $s\S$ which participates in the diagram
\[ \Z^{op} \xra{F^{op}} \Z^{op} \xra{\Nervei((-)(x,y))} s\S \]
defining $\colim_{\word{m} \in \Z^{op}} \Nervei(F(\word{m})(x,y))$.  So, in order to witness the commutativity of the square, it suffices to obtain an equivalence between the two maps
\[ \left| \Nervei((\varphi^{op}_\word{3})(x,y)) \right| , \left| \Nervei((F(\gamma))(x,y)) \right| \in \hom_\S \left( \left| \Nervei(\word{3}(x,y)) \right| , \left| \Nervei(F(\word{3})(x,y)) \right| \right) . \]
But there is an evident cospan in $\Funp(F(\word{3}) , \word{3})^\bW$ between the two maps $\varphi_\word{3}$ and $F(\gamma)$, so this follows from \cref{nat w.e. induces nat trans}, \cref{rnerves:nat trans induces equivce betw maps on gpd-complns}, and \cref{rnerves:groupoid-completion of CSSs}.

\item The fact that $| \rho | \circ |\ul{\varphi}| \simeq \id_{\left| \colim_{\word{m} \in \Z^{op}} \Nervei(\word{m}(x,y)) \right| }$ follows from applying \cref{gr:triangle of colimits} to the diagram
\[ \begin{tikzcd}[column sep=2cm]
\Z^{op} \arrow[bend left=50]{r}{\id_{\Z^{op}}}[swap, pos=0.2, transform canvas={yshift=-1em}]{\left. \varphi^{op} \right\Downarrow} \arrow[bend right=50]{r}[swap]{F^{op}} & \Z^{op} \arrow{r}{((-)(x,y))^\gpd} & \S
\end{tikzcd} \]
and invoking \cref{rnerves:groupoid-completion of CSSs} to obtain a retraction diagram
\[ \begin{tikzcd}[row sep=0.5cm]
\colim((-)^\gpd \circ \Nervei((-)(x,y)) \circ \id_{\Z^{op}} ) \arrow{rd}[sloped]{\sim} \arrow{dd}[swap]{|\ul{\varphi}|} \\
& \colim((-)^\gpd \circ \Nervei((-)(x,y)) ) . \\
\colim((-)^\gpd \circ \Nervei((-)(x,y)) \circ F^{op} ) \arrow{ru}[swap, sloped]{|\rho|}
\end{tikzcd} \]

\item\label{step using two results on coends}
 It is a straightforward exercise to check that for any $\word{m}' \in \Z$, the map
\[ \hom_{\Z}(\word{3} , \word{m}') \simeq \hom_{\Z}(G([\any]) , \word{m}' ) \ra \colim_{\word{m} \in \Z^{op}} \hom_{\Z} ( G(\word{m}) , \word{m}') \]
is an isomorphism: in other words, the map
\[ \hom_{\Z}(\word{3} , -) \ra \colim_{\word{m} \in \Z^{op}} \hom_{\Z}(G(\word{m}) , - ) \]
is an equivalence in $\Fun(\Z,\Set) \subset \Fun(\Z,\S)$.  Using this, and denoting by $- \tensoring - : s\S \times \S \ra s\S$ the evident tensoring
\[ s\S \times \S \xra{\id_{s\S} \times \const} s\S \times s\S \xra{- \times - } s\S , \]
we obtain the map
\[ \Nervei(\word{3}(x,y)) \xra{\beta} \colim_{\word{m} \in \Z^{op}} \Nervei(G(\word{m})(x,y)) \]
as string of equivalences
\begin{align*}
\Nervei ( \word{3} (x,y))
& \simeq \int^{\word{m}' \in \Z} \Nervei( \word{m}'(x,y)) \tensoring \hom_{\Z}(\word{3},\word{m}') \\
& = \int^{\Z} \Nervei((-)(x,y)) \tensoring \hom_{\Z}(\word{3} , -) \\
& \xra{\sim} \int^{\Z} \Nervei( (-) (x,y)) \tensoring \left( \colim_{\word{m} \in \Z^{op}}^{\Fun(\Z,\S)} \hom_{\Z}(G(\word{m}) , - ) \right) \\
& \simeq \colim_{\word{m} \in \Z^{op}}^{s\S} \left( \int^{\Z} \Nervei((-)(x,y)) \tensoring \hom_{\Z}(G(\word{m}) , - ) \right) \\
& = \colim_{\word{m} \in \Z^{op}}^{s\S} \left( \int^{\word{m}' \in \Z} \Nervei(\word{m}'(x,y)) \tensoring \hom_{\Z}(G(\word{m}) , \word{m}') \right) \\
& \simeq \colim_{\word{m} \in \Z^{op}}^{s\S} \Nervei(G(\word{m})(x,y)) \\
\end{align*}
in $s\S$, in which
\begin{itemize}
\item the second and fifth lines are purely for notational convenience,
\item we apply to the functor
\[ \Z^{op} \xra{\Nervei((-)(x,y))} s\S \]
\begin{itemize}
\item \cref{ninja yoneda} to obtain the first line,
\item \cref{wtd colimit is left adjoint in the weight} to obtain the fourth line, and
\item \cref{ninja yoneda} again to obtain the last line,
\end{itemize}
and
\item the third line follows from the equivalence in $\Fun(\Z,\S)$ obtained above.
\end{itemize}

the first and last lines are obtained from \cref{ninja yoneda} and the fourth line is obtained from \cref{wtd colimit is left adjoint in the weight}, all applied to the functor
\[ \Z^{op} \xra{\Nervei((-)(x,y))} s\S . \]
(So in fact, the map $\beta$ itself is already an equivalence in $s\S$ (i.e.\! before geometric realization).)

\item We claim that for every $\word{m} \in \Z^{op}$ the map
\[ \Nervei(G(\word{m})(x,y)) \xra{\Nervei((\psi^{op}_\word{m})(x,y))} \Nervei(F(\word{m})(x,y)) \]
in $s\S$ becomes an equivalence after geometric realization.  This follows from an analysis of the corepresenting map $F(\word{m}) \xra{\psi_{\word{m}}} G(\word{m})$ in $\Z \subset \RelCati$: it can be obtained as a composite
\[ F(\word{m}) = \word{m}'_0 \ra \word{m}'_1 \ra \cdots \ra \word{m}'_{|\word{m}|_{\bW^{-1}}-1} \ra \word{m}'_{|\word{m}|_{\bW^{-1}}} = G(\word{m}) \]
in $\Z$, in which each $\word{m}'_i$ is obtained from $\word{m}'_{i-1}$ by omitting one of the internal appearances of $\bW^{-1}$ in $F(\word{m})$, and the corresponding map $\word{m}'_i \ra \word{m}'_{i+1}$ is obtained by collapsing this copy of $\bW^{-1}$ to an identity map.  Each map
\[ \Nervei(\word{m}'_i(x,y)) \ra \Nervei(\word{m}'_{i-1}(x,y)) \]
in $s\S$ becomes an equivalence after geometric realization, by \cref{compose w.e.'s} when the about-to-be-omitted appearance of $\bW^{-1}$ in $\word{m}'_{i-1}$ is adjacent to another appearance of $\bW^{-1}$, and by applying the definition of $(\R,\bW)$ admitting a homotopical three-arrow calculus (\cref{define calculus}) to (either one or two iterations, depending on the shape of $\word{m}'_{i-1}$, of) the combination of \cref{concatenated zigzags} and \cref{gr:invce of two-sided Gr}.  Hence, the composite map 
\[ \Nervei(G(\word{m})(x,y)) = \Nervei(\word{m}'_{|\word{m}|_{\bW^{-1}}}(x,y)) \ra \cdots \ra \Nervei(\word{m}'_0(x,y)) = \Nervei(F(\word{m})(x,y)) , \]
which is precisely the map $\Nervei((\psi^{op}_\word{m})(x,y))$, does indeed become an equivalence upon geometric realization as well.  Then, since colimits commute, it follows that the induced map
\[ \left| \colim_{\word{m}' \in \Z^{op}} \Nervei(G(\word{m}')(x,y)) \right| \xra{|\ul{\psi}|} \left| \colim_{\word{m}' \in \Z^{op}} \Nervei(F(\word{m}')(x,y)) \right| \]
is an equivalence in $\S$. \qedhere

\end{enumerate}

\end{proof}

\section{Hammock localizations of relative $\infty$-categories}\label{section hammocks}

In \cref{section zigzags and hammocks}, given a relative $\infty$-category $(\R,\bW)$ and a pair of objects $x,y \in \R$, we defined the corresponding hammock simplicial space
\[ \homhamR(x,y) \in s\S \]
(see \cref{define hammocks}).  In this section, we proceed to \textit{globalize} this construction, assembling the various hammock simplicial spaces of $(\R,\bW)$ into a Segal simplicial space -- and thence a $s\S$-enriched $\infty$-category -- whose compositions encode the \textit{concatenation} of zigzags in $(\R,\bW)$.

The bulk of the construction of the hammock localization consists in constructing the \textit{pre}-hammock localization: this will be a Segal simplicial space
\[ \preham(\R,\bW) \in \SsS \subset s(s\S) , \]
whose $n\th$ level is given by the colimit
\[ {\colim}^{s\S}_{(\word{m}_1,\ldots,\word{m}_n) \in (\Z^{op})^{\times n}} \Nervei \left( \Fun( [\word{m}_1; \ldots ; \word{m}_n] , \R)^\bW \right) . \]
For clarity, we proceed in stages.

First, we build an object which simultaneously corepresents
\begin{itemizesmall}
\item all possible sequences (of any length) of composable zigzags, and
\item all possible concatenations among these sequences.
\end{itemizesmall}

\begin{constr}\label{main constrn for hammock localizn}
Observe that $\Z \in \Cat$ is a monoid object, i.e.\! a monoidal category: its multiplication is given by the concatenation functor
\[ \Z \times \Z \xra{[-;-]} \Z , \]
and the unit map $\pt_\Cat \ra \Z$ selects the terminal object $[\es] \in \Z$.\footnote{In fact, we can even consider $\Z$ as a monoid object in $\strcat$ (i.e.\! a \textit{strict} monoidal category), but this is unnecessary for our purposes.}  We can thus define its bar construction
\[ \bD^{op} \xra{\Bar(\Z)_\bullet} \Cat , \]
which has $\Bar(\Z)_n = \Z^{\times n}$ (so that $\Bar(\Z)_0 = \Z^{\times 0} = \pt_\Cat$), with face maps given by concatenation and with degeneracy maps given by the unit.  This admits an \textit{oplax} natural transformation to the functor
\[ \bD^{op} \xra{\const(\RelCat)} \Cat , \]
which we encode as a commutative triangle
\[ \begin{tikzcd}
\Grop(\Bar(\Z)_\bullet) \arrow{rr} \arrow{rd} & & \RelCat \times \bD \arrow{ld} \\
& \bD
\end{tikzcd} \]
in $\Cat$ (recall \cref{gr:define op/lax nat trans betw fctrs to Cati} and \cref{gr:Gr of a constant functor}): in simplicial degree $n$, this is given by the iterated concatenation functor
\[ \Bar(\Z)_n = \Z^{\times n} \xra{[-;\cdots;-]} \Z \hookra \RelCatp \ra \RelCat \]
(which in degree 0 is simply the composite
\[ \{ [\es] \} \hookra \RelCatp \ra \RelCat , \]
i.e.\! the inclusion of the terminal object $\{ \pt_\RelCat \} \hookra \RelCat$).\footnote{The reason that we must compose with the forgetful functor $\RelCatp \ra \RelCat$ is that the oplax structure maps (e.g.\! the inclusion $\word{m}_1 \hookra [\word{m}_1;\word{m}_2]$) do not respect the double-pointings.}\footnote{It is also true that for a monoidal ($\infty$-)category $\C$ whose unit object is terminal, the bar construction $\Bar(\C)_\bullet$ admits a canonical \textit{lax} natural transformation to $\const(\C)$, whose components are again given by the iterated monoidal product.  But this is distinct from what we seek here.}  Taking opposites, we obtain a commutative triangle
\[ \begin{tikzcd}
\Gr(\Bar(\Z^{op})_\bullet) \arrow{rr} \arrow{rd} & & \RelCat^{op} \times \bD^{op} \arrow{ld} \\
& \bD^{op}
\end{tikzcd} \]
in $\Cat$, which now encodes a \textit{lax} natural transformation from the bar construction
\[ \bD^{op} \xra{\Bar(\Z^{op})_\bullet} \Cat \]
on the monoid object $\Z^{op} \in \Cat$ (note that the involution $(-)^{op} : \Cat \xra{\sim} \Cat$ is covariant) to the functor
\[ \bD^{op} \xra{\const(\RelCat^{op})} \Cat . \]
\end{constr}

We now map into an arbitrary relative $\infty$-category and extract the indicated colimits, all in a functorial way.

\begin{constr}\label{map all-sequences-of-zigzags-at-once corepresenter into a rel infty-cat}
A relative $\infty$-category $(\R,\bW)$ represents a composite functor
\[ \RelCat \hookra \RelCati \xra{\Fun(-,\R)^\bW} \Cati \xra[\sim]{\Nervei} \CSS \hookra s\S . \]
Considering this as a natural transformation $\const(\RelCat^{op}) \ra \const(s\S)$ in $\Fun(\bD^{op},\Cati)$, we can postcompose it with the lax natural transformation obtained in \cref{main constrn for hammock localizn}, yielding a composite lax natural transformation encoded by the diagram
\[ \begin{tikzcd}
\Gr(\Bar(\Z^{op})_\bullet) \arrow{rr} \arrow{rrd} & & \RelCat^{op} \times \bD^{op} \arrow{rrrr}{\Nervei \left( \Fun(-,\R)^\bW \right) \times \id_{\bD^{op}}} \arrow{d} & & & & s\S \times \bD^{op} \arrow{lllld} \\
& & \bD^{op}
\end{tikzcd} \]
in $\Cati$.  Then, by Proposition T.4.2.2.7, there is a unique ``fiberwise colimit'' lift in the diagram
\[ \begin{tikzcd}
\Gr(\Bar(\Z^{op})_\bullet) \arrow{r} \arrow{d} & s\S \times \bD^{op} \arrow{d} \\
\Gr(\Bar(\Z^{op})_\bullet) \underset{\bD^{op}}{\diamond} \bD^{op} \arrow{r} \arrow[dashed]{ru} & \bD^{op}
\end{tikzcd} \]
in $\Cati$.\footnote{The object in the bottom left of this diagram is a ``relative join'' (see Definition T.4.2.2.1), which in this case actually simply reduces to a ``directed mapping cylinder'' (see \cref{gr:ex cocart over walking arrow}).}  Thus, the resulting composite
\[ \bD^{op} \ra \Gr(\Bar(\Z^{op})_\bullet) \underset{\bD^{op}}{\diamond} \bD^{op} \ra s\S \times \bD^{op} \ra s\S \]
takes each object $[n]^\opobj \in \bD^{op}$ to the colimit of the composite
\[ \Bar(\Z^{op})_n = (\Z^{op})^{\times n} \xra{[-;\cdots;-]^{op}} \Z^{op} \hookra (\RelCatp)^{op} \ra \RelCat^{op} \xra{\Nervei \left( \Fun(-,\R)^\bW \right)} s\S . \]
We denote this simplicial object in simplicial spaces by
\[ \bD^{op} \xra{\preham(\R,\bW)} s\S . \]
Allowing $(\R,\bW) \in \RelCati$ to vary, this assembles into a functor
\[ \RelCati \xra{\preham} s(s\S) . \]
\end{constr}

We now show that the bisimplicial spaces of \cref{map all-sequences-of-zigzags-at-once corepresenter into a rel infty-cat} are in fact Segal simplicial spaces.

\begin{lem}\label{preham is Segal}
For any $(\R,\bW) \in \RelCati$, the object $\preham(\R,\bW) \in s(s\S)$ satisfies the Segal condition.
\end{lem}

\begin{proof}
We must show that for every $n \geq 2$, the $n\th$ Segal map
\[ \preham(\R,\bW)_n \ra \preham(\R,\bW)_1 \underset{t,\preham(\R,\bW)_0,s}{\times} \cdots \underset{t,\preham(\R,\bW)_0,s}{\times} \preham(\R,\bW)_1 \]
(to the $n$-fold fiber product) is an equivalence in $s\S$.  As $s\S$ is an $\infty$-topos, colimits therein are universal, i.e.\! they commute with pullbacks (see Definition T.6.1.0.4 and Theorem T.6.1.0.6 (and the discussion at the beginning of \sec T.6.1.1)).  Moreover, note that we have a canonical equivalence $\preham(\R,\bW)_0 \simeq \Nervei(\bW)$ in $s\S$.  Hence, by induction, we have a string of equivalences
\begin{align*}
& \preham(\R,\bW)_1 \underset{t,\preham(\R,\bW)_0,s}{\times} \cdots \underset{t,\preham(\R,\bW)_0,s}{\times} \preham(\R,\bW)_1 \\
& \simeq \preham(\R,\bW)_1 \underset{\{1\} , \preham(\R,\bW)_0, \{0\}}{\times} \preham(\R,\bW)_{n-1} \\
& = \lim \left( \begin{tikzcd}[ampersand replacement=\&]
\& \colim_{(\word{m}_2,\ldots,\word{m}_n) \in (\Z^{op})^{\times (n-1)}} \Nervei \left( \Fun([\word{m}_2;\ldots;\word{m}_n],\R)^\bW \right) \arrow{d} \\
\colim_{\word{m}_1 \in \Z^{op}} \Nervei \left( \Fun(\word{m}_1,\R)^\bW \right) \arrow{r} \& \Nervei(\bW)
\end{tikzcd} \right) \\
& \simeq \colim_{\word{m}_1 \in \Z^{op}} \left(
\lim \left( \begin{tikzcd}[ampersand replacement=\&]
\& \colim_{(\word{m}_2,\ldots,\word{m}_n) \in (\Z^{op})^{\times (n-1)}} \Nervei \left( \Fun([\word{m}_2;\ldots;\word{m}_n],\R)^\bW \right) \arrow{d} \\
\Nervei \left( \Fun(\word{m}_1,\R)^\bW \right) \arrow{r} \& \Nervei(\bW)
\end{tikzcd} \right)
\right) \\
& \simeq \colim_{\word{m}_1 \in \Z^{op}} \left(
\colim_{(\word{m}_2,\ldots,\word{m}_n) \in (\Z^{op})^{\times (n-1)}} \left(
\lim \left( \begin{tikzcd}[ampersand replacement=\&]
\& \Nervei \left( \Fun([\word{m}_2;\ldots;\word{m}_n],\R)^\bW \right) \arrow{d} \\
\Nervei \left( \Fun(\word{m}_1,\R)^\bW \right) \arrow{r} \& \Nervei(\bW)
\end{tikzcd} \right)
\right)
\right) \\
& \simeq \colim_{(\word{m}_1,\ldots,\word{m}_n) \in (\Z^{op})^{\times n}} \Nervei \left( \Fun([\word{m}_1;\ldots;\word{m}_n],\R)^\bW \right) \\
& = \preham(\R,\bW)_n
\end{align*}
(where in the penultimate line we appeal to Fubini's theorem for colimits) which, chasing through the definitions, visibly coincides with the $n\th$ Segal map.  This proves the claim.
\end{proof}

We finally come to the main point of this section.

\begin{defn}\label{define hammock localizn}
By \cref{preham is Segal}, the functor given in \cref{map all-sequences-of-zigzags-at-once corepresenter into a rel infty-cat} admits a factorization
\[ \begin{tikzcd}
\RelCati \arrow{r}{\preham} \arrow[dashed]{rd} & s(s\S) \\
& \SsS \arrow[hook]{u}
\end{tikzcd} \]
through the $\infty$-category of Segal simplicial spaces.  We again denote this factorization by
\[ \RelCati \xra{\preham} \SsS , \]
and refer to it as the \bit{pre-hammock localization} functor.\footnote{The terminology ``pre-hammock localization'' should be parsed as ``pre-(hammock localization)'': it already contains the hammock simplicial spaces (see \cref{justify hom notation}), it is just not itself the hammock localization.}  Then, we define the \bit{hammock localization} functor
\[ \RelCati \xra{\ham} \CatsS \]
to be the composite
\[ \RelCati \xra{\preham} \SsS \xra{\spat(-)} \CatsS . \]
\end{defn}

\begin{rem}\label{justify hom notation}
Given a relative $\infty$-category $(\R,\bW)$, the $0\th$ level of its pre-hammock localization
\[ \preham(\R,\bW) \in \SsS \subset s(s\S) \]
is given by
\[ \colim \left( \{ [\es] \}^\opobj \hookra (\RelCatp)^{op} \ra \RelCat^{op} \xra{\Nervei \left( \Fun( - \R )^\bW \right) } s \S \right) , \]
which is simply the nerve $\Nervei(\bW) \in s\S$ of the subcategory $\bW \subset \R$ of weak equivalences.  Thus, its space of objects is simply
\[ \preham(\R,\bW)_0 \simeq \Nervei(\bW)_0 \simeq \bW^\simeq \simeq \R^\simeq . \]
Moreover, unwinding the definitions, it is manifestly clear that
\begin{itemizesmall}
\item its hom-simplicial spaces are precisely the hammock simplicial spaces of $(\R,\bW)$ (recall Definitions \ref{define space of objects and hom-sspaces} \and \ref{define hammocks}), and
\item its compositions correspond to concatenation of zigzags (with identity morphisms corresponding to zigzags of type $[\es] \in \Z$).
\end{itemizesmall}
Of course, we have a canonical counit weak equivalence
\[ \ham(\R,\bW) \we \preham(\R,\bW) \]
in $\SsS_\DK$ which is even fully faithful in the $s\S$-enriched sense, so that the hammock localization enjoys all these same properties.
\end{rem}

Just as in the 1-categorical case, the hammock localization of $(\R,\bW)$ admits a natural map from $\R$.

\begin{constr}\label{construct hammock localizn map}
Returning to \cref{main constrn for hammock localizn}, observe that there is a tautological section
\[ \begin{tikzcd}
\Grop(\Bar(\Z)_\bullet) \arrow{d} \\
\bD \arrow[dashed, bend left]{u}
\end{tikzcd} \]
which takes $[n] \in \bD$ to $([\any],\ldots,[\any]) \in \Z^{\times n} = \Bar(\Z)_n$, and which takes a map $[m] \xra{\varphi} [n]$ in $\bD$ to the map corresponding to the fiber map which, in the $i\th$ factor of $\Z^{\times m}$, is given by the unique map
\[ [ \any ] \ra [ \any^{\circ ( \varphi(i) - \varphi(i-1) ) } ] \]
in $\Z$.  This is opposite to a tautological section
\[ \begin{tikzcd}
\Gr(\Bar(\Z^{op})_\bullet) \arrow{d} \\
\bD^{op} \arrow[dashed, bend left]{u}
\end{tikzcd} \]
which gives rise to a composite map
\[ \bD^{op} \ra \Gr(\Bar(\Z^{op})_\bullet) \ra \Gr(\Bar(\Z^{op})_\bullet) \underset{\bD^{op}}{\diamond} \bD^{op} \]
admitting a natural transformation to the standard inclusion (as the ``target'' factor, i.e.\! the fiber over $1 \in [1]$).  This postcomposes with the composite
\[ \Gr(\Bar(\Z^{op})_\bullet) \underset{\bD^{op}}{\diamond} \bD^{op} \ra s\S \times \bD^{op} \ra s\S \]
appearing in \cref{map all-sequences-of-zigzags-at-once corepresenter into a rel infty-cat} to give a natural transformation
\[ \Nervei^\lw \left( \Fun([\bullet],\R)^\bW \right) \ra \preham(\R,\bW)_\bullet \]
in $\Fun(\bD^{op},s\S)$.\footnote{Note that this source is just the image of the Rezk pre-nerve $\preNerveRezki(\R,\bW)_\bullet \in s\Cati$ under the inclusion $s\Cati \xra{\sim} s\CSS \hookra s(s\S)$ (recall \cref{rnerves:define infty-categorical rezk pre-nerve and rezk nerve}).}  Thus, in simplicial degree $n$, this map is simply the inclusion into the colimit defining $\preham(\R,\bW)_n \in s\S$ at the object
\[ ([\any]^\opobj , \ldots , [\any]^\opobj) \in (\Z^{op})^{\times n} . \]
Restricting levelwise to (the nerve of) the maximal subgroupoid, we obtain a composite
\begin{align*}
\const(\R)_\bullet & = \const^\lw ( \forget_\CSS ( \Nervei(\R)))_\bullet \\
& = \const^\lw ( \hom_\Cati ( [\bullet],\R)) \\
& \simeq \const^\lw \left( \Fun([\bullet],\R)^\simeq \right) \\
& \simeq \Nervei^\lw \left( \Fun([\bullet],\R)^\simeq \right) \\
& \hookra \Nervei^\lw \left( \Fun([\bullet],\R)^\bW \right) \\
& \ra \preham(\R,\bW)_\bullet .
\end{align*}
As this source lies in $\CatsS \subset \SsS$, we obtain a canonical factorization
\[ \begin{tikzcd}
\const(\R) \arrow{r} \arrow[dashed]{rd} & \preham(\R,\bW) \\
& \ham(\R,\bW) \arrow{u}[sloped, anchor=north]{\approx}
\end{tikzcd} \]
in $(\CatsS)_\DK$.  This clearly assembles into a natural transformation
\[ \const \ra \ham \]
in $\Fun(\RelCati , \CatsS)$.
\end{constr}

\begin{defn}
For a relative $\infty$-category $(\R,\bW)$, we refer to the map
\[ \const(\R) \ra \ham(\R,\bW) \]
in $\CatsS$ of \cref{construct hammock localizn map} as its \bit{tautological inclusion}.
\end{defn}

We end this section with the following fundamental result, an analog of \cite[Proposition 3.3]{DKCalc}; roughly speaking, it shows that when considered as morphisms in the hammock localization, weak equivalences in $\R$ both represent and corepresent equivalences in the underlying $\infty$-category.  Just as with the fundamental theorem of homotopical three-arrow calculi (\ref{calculus gives reduction}), its proof will be substantially more involved than that of its 1-categorical analog (recall \cref{intro rem more details than DK}).


\begin{prop}\label{hammocks are invt under w.e.}
Let $(\R,\bW) \in \RelCati$ let $r,y,z \in \R$.  Suppose we are given a weak equivalence
\[ w \in \hom_\bW(y,z) \subset \hom_\R(y,z) , \]
and let us also denote by $w \in \homhamR(y,z)_0$ the resulting composite morphism
\[ \pt_{s\S} \ra [\any](y,z) \ra \homhamR(y,z) . \]
Then, the induced composite maps
\[ \begin{tikzcd}[column sep=4cm, row sep=1.5cm]
\homhamR(r,y) \simeq \homhamR(r,y) \times \pt_{s\S} \arrow{r}{\id_{\homhamR(r,y)} \times w} & \homhamR(r,y) \times \homhamR(y,z) \arrow{d}{\chi^{\ham(\R,\bW)}_{r,y,z}} \\
&  \homhamR(r,z)
\end{tikzcd} \]
and
\[ \begin{tikzcd}[column sep=4cm, row sep=1.5cm]
\homhamR(z,r) \simeq \pt_{s\S} \times \homhamR(z,r) \arrow{r}{w \times \id_{\homhamR(z,r)}} & \homhamR(y,z) \times \homhamR(z,r) \arrow{d}{\chi^{\ham(\R,\bW)}_{y,z,r}} \\
& \homhamR(y,r)
\end{tikzcd} \]
in $s\S$ become equivalences in $\S$ upon application of the geometric realization functor $|{-}| : s\S \ra \S$.
\end{prop}

\begin{proof}
We prove the first statement; the second statement follows by a nearly identical argument.

Using the composite
\[ \pt_{s\S} \ra [\bW^{-1}](z,y) \ra \homhamR(z,y) , \]
which morphism we will denote by $w^{-1} \in \homhamR(z,y)_0$, we can form the map
\[ \begin{tikzcd}[column sep=4cm, row sep=1.5cm]
\homhamR(r,z) \simeq \homhamR(r,z) \times \pt_{s\S} \arrow{r}{\id_{\homhamR(r,z)} \times w^{-1}} & \homhamR(r,z) \times \homhamR(z,y) \arrow{d}{\chi^{\ham(\R,\bW)}_{r,z,y}} \\
&  \homhamR(r,y)
\end{tikzcd} \]
in $s\S$.  We claim that upon geometric realization, this gives an inverse of the map
\[ \left| \homhamR(r,y) \right| \ra \left| \homhamR(r,z) \right| \]
in $\S$.  We will only show that the composite map
\[ \left| \homhamR(r,y) \right| \ra \left| \homhamR(r,z) \right| \ra \left| \homhamR(r,y) \right| \]
is an equivalence; that the composite 
\[ \left| \homhamR(r,z) \right| \ra \left| \homhamR(r,y) \right| \ra \left| \homhamR(r,z) \right| \]
is an equivalence will follow from a very similar argument.

For each $\word{m} \in \Z^{op}$, let us define a functor
\[ \word{m}(r,y) \xra{\varphi_{\word{m}}} [\word{m};\any;\bW^{-1}](r,y) \]
given informally by taking a zigzag
\[ \begin{tikzcd} r \arrow[-]{rr}{\word{m}} & & y \end{tikzcd} \]
in $(\R,\bW)$ to the zigzag
\[ \begin{tikzcd} r \arrow[-]{rr}{\word{m}} & & y \arrow{r} & z & y \arrow{l}[swap]{\approx} \end{tikzcd} \]
in $(\R,\bW)$, in which both new maps are the chosen weak equivalence $w$.\footnote{This (and subsequent constructions) can easily be made precise by defining a suitable notion of a map in a relative word being forced to land at $w$; we will leave such a precise construction to the interested reader.}  This operation is clearly natural in $\word{m} \in \Z^{op}$, i.e.\! it assembles into a natural transformation
\[ \begin{tikzcd}
\Z^{op} \arrow[bend left=80, out=80, in=100]{rr}{(-)(r,y)}[swap, pos=0.4, transform canvas={yshift=-3em}]{\left. \varphi \right\Downarrow} \arrow[bend right]{rd}[swap]{[-;\any;\bW^{-1}]} & & \Cati . \\
& \Z^{op} \arrow[bend right]{ru}[swap]{(-)(r,y)}
\end{tikzcd} \]
Then, using \cref{rnerves:groupoid-completion of CSSs} and the fact that the geometric realization functor $s\S \xra{|{-}|} \S$ commutes with colimits (being a left adjoint), we see that the composite
\[ \left| \homhamR(r,y) \right| \ra \left| \homhamR(r,z) \right| \ra \left| \homhamR(r,y) \right| \]
is obtained as the composite
\[ \begin{tikzcd}[column sep=4cm, row sep=1.5cm]
\colim_{\Z^{op}} \left( (-)^\gpd \circ (-)(r,y) \right) \arrow{d}[swap]{\colim_{\Z^{op}}(\id_{(-)^\gpd} \circ \varphi)} \\
\colim_{\Z^{op}} \left( (-)^\gpd \circ (-)(r,y) \circ [- ; \any ; \bW^{-1}] \right) \arrow{r}[swap]{\colim^\S([-;\any;\bW^{-1}])} & \colim_{\Z^{op}} \left( (-)^\gpd \circ (-)(r,y) \right) .
\end{tikzcd} \]
To see that this is an equivalence, for each $\word{m} \in \Z^{op}$ let us define a map $\word{m} \xra{\psi_\word{m}} [\word{m};\any;\bW^{-1}]$ in $\Z^{op}$ to be opposite the map $[\word{m};\any;\bW^{-1}] \ra \word{m}$ in $\Z$ which collapses the newly concatenated copy of $[\any;\bW^{-1}]$ to the map $\id_{t_\word{m}}$.  These assemble into a natural transformation $\id_{\Z^{op}} \xra{\psi} [-;\any;\bW^{-1}]$ in $\Fun(\Z^{op},\Z^{op})$, and hence we obtain a natural transformation
\[ \begin{tikzcd}[column sep=2cm, row sep=1.5cm]
\Z^{op} \arrow[bend left=80, out=80, in=100]{rr}{(-)(r,y)}[swap, pos=0.21, transform canvas={yshift=-3.1em}]{\left. \id_{(-)(r,y)} \circ \psi \right\Downarrow} \arrow[bend right]{rd}[swap]{[-;\any;\bW^{-1}]} & & \Cati . \\
& \Z^{op} \arrow[bend right]{ru}[swap]{(-)(r,y)}
\end{tikzcd} \]
Moreover, For each $\word{m} \in \Z^{op}$ we have a functor
\[ [1] \times \word{m}(r,y) \xra{\mu_\word{m}} [\word{m};\any;\bW^{-1}](r,y) , \]
adjoint to a functor
\[ \word{m}(r,y) \ra \Fun([1],[\word{m};\any;\bW^{-1}](r,y)) , \]
given informally by taking a zigzag
\[ \begin{tikzcd} r \arrow[-]{rr}{\word{m}} & & y \end{tikzcd} \]
in $(\R,\bW)$ to the diagram
\[ \begin{tikzcd}
r \arrow[equals]{d} \arrow[-]{rr}{\word{m}} & & y \arrow[equals]{d} \arrow{r} & y \arrow{d}[sloped, anchor=south]{\approx} & y \arrow{l}[swap]{\approx} \arrow[equals]{d} \\
r \arrow[-]{rr}[swap]{\word{m}} & & y \arrow{r} & z & y \arrow{l}{\approx}
\end{tikzcd} \]
in $(\R,\bW)$ representing a morphism in $[\word{m};\any;\bW^{-1}](r,y)$, where the maps in the right two squares are all either the chosen weak equivalence $y \we z$ or are $\id_y$.  These assemble into a morphism
\[ \const([1]) \times (-)(r,y) \xra{\mu} (-)(r,y) \circ [-;\any;\bW^{-1}] \]
in $\Fun(\Z^{op},\Cati)$, i.e.\! a modification from $\id_{(-)(r,y)} \circ \psi$ to $\varphi$.  By \cref{gr:mod gives nat trans}, this induces a natural transformation
\[ \begin{tikzcd}[column sep=2cm]
\Gr((-)(r,y)) \arrow[bend left=50]{r}{\Gr(\id_{(-)(r,y)} \circ \psi)}[swap, pos=0.32, transform canvas={yshift=-2.8em}]{\left. \Gr(\mu) \right\Downarrow} \arrow[bend right=50]{r}[swap]{\Gr(\varphi)} & \Gr((-)(r,y) \circ [-;\any;\bW^{-1}])
\end{tikzcd} \]
which, by \cref{rnerves:nat trans induces equivce betw maps on gpd-complns} and \cref{gr:groupoid-completion of the grothendieck construction}, gives a homotopy between the maps
\[ \colim_{\Z^{op}} \left( (-)^\gpd \circ (-)(r,y) \right)
\xra{\colim_{\Z^{op}}( \id_{(-)^\gpd} \circ \id_{(-)(r,y)} \circ \psi)}
\colim_{\Z^{op}} \left( (-)^\gpd \circ (-)(r,y) \circ [-;\any;\bW^{-1}] \right) \]
and
\[ \colim_{\Z^{op}} \left( (-)^\gpd \circ (-)(r,y) \right)
\xra{\colim_{\Z^{op}}(\id_{(-)^\gpd} \circ \varphi)}
\colim_{\Z^{op}} \left( (-)^\gpd \circ (-)(r,y) \circ [-;\any;\bW^{-1}] \right) \]
in $\S$.  Hence, to show that the above composite is an equivalence, it suffices to show that the composite
\[ \begin{tikzcd}[column sep=4cm, row sep=1.5cm]
\colim_{\Z^{op}} \left( (-)^\gpd \circ (-)(r,y) \right) \arrow{d}[swap]{\colim_{\Z^{op}}( \id_{(-)^\gpd} \circ \id_{(-)(r,y)} \circ \psi)
 } \\
\colim_{\Z^{op}} \left( (-)^\gpd \circ (-)(r,y) \circ [- ; \any ; \bW^{-1}] \right) \arrow{r}[swap]{\colim^\S([-;\any;\bW^{-1}])} & \colim_{\Z^{op}} \left( (-)^\gpd \circ (-)(r,y) \right)
\end{tikzcd} \]
is an equivalence.  But this composite fits into a commutative triangle
\[ \begin{tikzcd}[row sep=0.5cm]
\colim_{\Z^{op}} ((-)^\gpd \circ (-)(r,y) \circ \id_{\Z^{op}}) \arrow{dd} \arrow{rd}[sloped, pos=0.45]{\sim} \\
& \colim_{\Z^{op}} ((-)^\gpd \circ (-)(r,y)) \\
\colim_{\Z^{op}} ((-)^\gpd \circ (-)(r,y) \circ [-;\any;\bW^{-1}]) \arrow{ru}
\end{tikzcd} \]
obtained by applying \cref{gr:triangle of colimits} to the diagram
\[ \begin{tikzcd}[column sep=1.5cm]
\Z^{op} \arrow[bend left=50]{r}{\id_{\Z^{op}}}[swap, pos=0.3, transform canvas={yshift=-1.2em}]{\left. \psi \right\Downarrow} \arrow[bend right=50]{r}[swap]{[-;\any;\bW^{-1}]} & \Z^{op} \arrow{r}{(-)(r,y)} & \Cati ,
\end{tikzcd} \]
so it is an equivalence.  This proves the claim.
\end{proof}

\begin{rem}\label{und infty-cat of ham inverts weak equivces}
By Yoneda's lemma, \cref{hammocks are invt under w.e.} implies that (the morphisms corresponding to) weak equivalences in $\ham(\R,\bW)$ become equivalences in the underlying $\infty$-category
\[ \Nervei^{-1} \left( \leftloc_\CSS \left( \left| \ham(\R,\bW) \right| \right) \right)  \in \Cati , \]
i.e.\! upon application of the composite
\[ \CatsS \xra{|{-}|} \SS \xra{\leftloc_\CSS} \CSS \xra[\sim]{\Nervei^{-1}} \Cati . \]
It follows that there exists a unique factorization
\[ \begin{tikzcd}
\R \arrow{r} \arrow{d} & \Nervei^{-1} \left( \leftloc_\CSS \left( \left| \ham(\R,\bW) \right| \right) \right) \\
\loc{\R}{\bW} \arrow[dashed]{ru}
\end{tikzcd} \]
of the image of the tautological inclusion $\const(\R) \ra \ham(\R,\bW)$ in $\CatsS$.
\end{rem}

\section{From fractions to complete Segal spaces, redux}\label{section fractions redux}

As an application of the theory developed in this paper, we now provide a \textit{sufficient condition} for the Rezk nerve
$ \NerveRezki(\R,\bW) \in s\S $
of a relative $\infty$-category $(\R,\bW)$ to be either
\begin{itemizesmall}
\item a Segal space or
\item a complete Segal space,
\end{itemizesmall}
thus giving a partial answer to our own \cref{rnerves:ask for CSS}, which refer to as the \bit{calculus theorem}.\footnote{The Rezk nerve is a straightforward generalization of Rezk's ``classification diagram'' construction, which we introduced and studied in \cref{rnerves:section rezk nerve}.}  This result is itself a direct generalization of joint work with Low regarding relative 1-categories (see \cite[Theorem 4.11]{LowMG}).  That result, in turn, generalizes work of Rezk, Bergner, and Barwick--Kan; we refer the reader to \cite[\sec 1]{LowMG} for a more thorough history.


\begin{thm}\label{calculus result}
Suppose that $(\R,\bW) \in \RelCati$ admits a homotopical three-arrow calculus.
\begin{enumerate}
\item\label{calculus for SS} $\NerveRezki(\R,\bW) \in s\S$ is a Segal space.
\item\label{calculus for CSS} Suppose moreover that $\bW \subset \R$ satisfies the two-out-of-three property.  Then $\NerveRezki(\R,\bW) \in s\S$ is a complete Segal space if and only if $(\R,\bW)$ is saturated.
\end{enumerate}
\end{thm}

The proof of the calculus theorem (\ref{calculus result}) is very closely patterned on the proof of \cite[Theorem 4.11]{LowMG} (the main theorem of that paper), which is almost completely analogous but holds only for relative 1-categories.\footnote{The 1-categorical Rezk nerve and the Rezk nerve of a relative $\infty$-category are essentially equivalent (see \cref{rnerves:infty-catl rezk nerve agrees with 1-catl rezk nerve}), which is why essentially the same proof can be applied in both cases.}  We encourage any reader who would like to understand it to first read that paper: there are no truly new ideas here, only generalizations from 1-categories to $\infty$-categories.

\begin{proof}[Proof of \cref{calculus result}]
For this proof, we give a detailed step-by-step explanation of what must be changed in the paper \cite{LowMG} to generalize its main theorem from relative 1-categories to relative $\infty$-categories.
\begin{itemize}

\item For \cite[Definition 2.1]{LowMG}, we replace the notion of a ``weak homotopy equivalence'' of categories by the notion of a map in $\Cati$ which becomes an equivalence under $(-)^\gpd : \Cati \ra \S$ (i.e.\! a Thomason weak equivalence (see \cref{gr:define Thomason model str} and \cref{gr:Th w.e.'s created by gpd compln})).

\item The proof of \cite[Lemma 2.2]{LowMG} carries over easily using \cref{rnerves:nat trans induces equivce betw maps on gpd-complns}.

\item For \cite[Definition 2.3]{LowMG}, we replace the notion of a ``homotopy pullback diagram'' of categories by the notion of a commutative square in $\Cati$ which becomes a pullback square under $(-)^\gpd : \Cati \ra \S$ (i.e.\! a homotopy pullback diagram in $(\Cati)_\Thomason$).

\item For \cite[Definition 2.4]{LowMG}, we replace the notions of ``Grothendieck fibrations'' and ``Grothendieck opfibrations'' of categories by those of cartesian fibrations and cocartesian fibrations of $\infty$-categories (see \cref{gr:section gr} and \cite{grjl}).

\item For \cite[Remark 2.5]{LowMG}, as the entire theory of $\infty$-categories is in essence already only pseudofunctorial, there is no corresponding notion of a co/cartesian fibration being ``split'' (or rather, \textit{every} co/cartesian fibration should be thought of as being ``split'').

\item The evident generalization of \cite[Example 2.6]{LowMG} can be obtained by applying Corollary T.2.4.7.12 to an identity functor of $\infty$-categories.

\item The evident generalization of (the first of the two dual statements of) \cite[Theorem 2.7]{LowMG} is proved as \cref{gr:pullback of cocart fibn}.

\item The evident generalization of \cite[Corollary 2.8]{LowMG} again follows directly (or can alternatively be obtained by combining \cref{rnerves:ex maximal localization} and \cref{rnerves:localization preserves finite products}).

\item For \cite[Definition 2.9]{LowMG}, we use the definition of the ``two-sided Grothendieck construction'' given in \cref{gr:define two-sided Gr}.  (Note that the 1-categorical version is simply the corresponding (strict) fiber product.)

\item The evident analog of \cite[Lemma 2.11]{LowMG} is proved as \cref{gr:invce of two-sided Gr}.

\item For \cite[Definition 3.1]{LowMG}, we replace the notion of a ``relative category'' by the notion of a ``relative $\infty$-category'' given in \cref{rnerves:define rel infty-cat}; recall from \cref{rnerves:weaker defn of rel infty-cat} that here we are actually working with a slightly weaker definition.  We replace the notion of its ``homotopy category'' by that of its localization given in \cref{rnerves:define localization}.  We have already defined the notion of a relative $\infty$-category being ``saturated'' in \cref{rnerves:define saturated}.

\item For \cite[Definition 3.2]{LowMG}, we have already made the analogous definitions in \cref{rnerves:define internal hom in relcats}.

\item For \cite[Definitions 3.3 and 3.6]{LowMG}, we have already made the analogous definitions in Definitions \ref{define relative word} \and \ref{define zigzags}.

\item The evident analog of \cite[Remark 3.7]{LowMG} is now true by definition (recall \cref{define enrichment and tensoring of doubly-pointed relcats over relcats}).

\item For \cite[Proposition 3.8]{LowMG}, the paper actually only uses part (ii), whose evident analog is provided by \cref{ev at free target gives cocart fibn}\ref{add W}.

\item For \cite[Lemma 3.10]{LowMG}, note that the functors in the statement of the result as well as in its proof are all corepresented by maps in $\RelCatp \subset \RelCatip$; the proof of the analogous result thus carries over by \cref{nat w.e. induces nat trans}.

\item For \cite[Lemma 3.11]{LowMG}, again everything in the statement of the result as well as in its proof are all corepresented; again the proof carries over by \cref{nat w.e. induces nat trans}.

\item For \cite[Definition 4.1]{LowMG}, we have already defined a ``homotopical three-arrow calculus'' for a relative $\infty$-category in \cref{define calculus}.

\item For \cite[Theorem 4.5]{LowMG}, we use the more general but slightly different definition of hammocks given in \cref{define hammocks} (recall \cref{hammocks disagree}); part (i) is proved as \cref{calculus gives reduction}, while part (ii) follows immediately from the definitions, particularly Definitions \ref{define hammock localizn} \and \ref{define space of objects and hom-sspaces}.  (Note that in the present framework, the ``reduction map'' is simply replaced by the canonical map to the colimit defining the simplicial space of hammocks.)

\item For \cite[Corollary 4.7]{LowMG}, the evident analog of \cite[Proposition 3.3]{DKCalc} is proved as \cref{hammocks are invt under w.e.}.

\item For \cite[Proposition 4.8]{LowMG}, the proof carries over essentially without change.  (The functor considered there when proving that the rectangle (AC) is a homotopy pullback diagram is replaced by our functor $\bW^{op} \xra{\word{3}(x,-)} \Cati$ of \cref{notation for half-doubly-pointed words}.)

\item For \cite[Lemma 4.9]{LowMG}, the map itself in the statement of the result comes from the functoriality
\[ \bW^{op} \xra{[\bW^{-1};\any^{\circ n};\bW^{-1}](x,-)} \Cati \]
and
\[ \bW \xra{[\bW^{-1};\any^{\circ n};\bW^{-1}](-,y)} \Cati \]
of \cref{notation for half-doubly-pointed words}, as do the vertical maps in the commutative square in the proof.  The horizontal maps in that square are corepresented by maps in $\Z \subset \RelCatp \subset \RelCatip$, and it clearly commutes by construction.  The evident analog of \cite[Proposition 9.4]{DKCalc} is proved as \cref{concatenated zigzags}.

\item For \cite[Proposition 4.10]{LowMG}, note that all morphisms in both the statement of the result and its proof are corepresented by maps in $\Z \subset \RelCatp \subset \RelCatip$; the proof itself carries over without change.

\item For \cite[Theorem 4.11]{LowMG} (whose analog is \cref{calculus result} itself), note that we are now proving an $\infty$-categorical statement (instead of a model-categorical one), and so there are no issues with fibrant replacement.

\begin{itemize}

\item

The proof of part \ref{calculus for SS} of \cref{calculus result} is identical to the proof of part (i) there: it follows from our analog of \cite[Proposition 4.10]{LowMG}.

\item

We address the two halves of the proof of part \ref{calculus for CSS} of \cref{calculus result} in turn.

\begin{itemize}

\item

The proof of the ``only if'' direction runs analogously to that of \cite[Theorem 4.11(ii)]{LowMG}, only now we use that given two objects $\pt_\Cati \rra \C$ in an $\infty$-category $\C$, any path between their postcompositions $\pt_\Cati \rra \C \ra \C^\gpd$ can be represented by a zigzag $\Nerve^{-1}(\sd^i(\Delta^1)) \ra \C$ connecting them (for some sufficiently large $i$).

\item

We must modify the proof of the ``if'' direction slightly, as follows.  Assume that $(\R,\bW) \in \RelCati$ is saturated.  By the local universal property of the Rezk nerve (\cref{rnerves:rezk nerve of a relative infty-category is initial}), we have an equivalence $\leftloc_\CSS(\NerveRezki(\R,\bW)) \simeq \Nervei(\loc{\R}{\bW})$ in $\CSS \subset s\S$.  Note also that by the two-out-of-three assumption, any two objects $\pt_\Cati \rra \Fun([1],\R)^\bW$ which select the same path component under the composite
\[ \pt_\Cati \rra \Fun([1],\R)^\bW \ra \left( \Fun([1],\R)^\bW \right)^\gpd = \NerveRezki(\R,\bW)_1 \]
are either both weak equivalences or both not weak equivalences.  Now, for any object of $\Fun([1],\R)^\bW$, recalling \cref{any SS pulled back from its CSS-localizn} and invoking the saturation assumption, we see that the corresponding map $[1] \ra \R$ selects an equivalence under the postcomposition
$[1] \ra \R \ra \loc{\R}{\bW}$
if and only if it factors as $[1] \ra \bW \hookra \R$.  From here, the proof proceeds identically.
\qedhere
\end{itemize}
\end{itemize}
\end{itemize}
\end{proof}

\begin{rem}
After establishing the necessary facts concerning model $\infty$-categories, we obtain an analog of \cite[Corollary 4.12]{LowMG} as \cref{fundthm:rnerve is a CSS}.
\end{rem}

\begin{rem}
In light of \cref{rnerves:infty-catl rezk nerve agrees with 1-catl rezk nerve}, \cite[Remark 4.13]{LowMG} is strictly generalized by the local universal property of the Rezk nerve (\cref{rnerves:rezk nerve of a relative infty-category is initial}).
\end{rem}

\bibliographystyle{amsalpha}
\bibliography{hammocks}{}

\end{document}